\documentclass[11pt]{amsart}  
\usepackage{amssymb}
\usepackage{amsfonts}
\usepackage{amsthm}
\usepackage{amsthm}
\usepackage{amscd}
\numberwithin{equation}{section}
\overfullrule=0pt \theoremstyle{plain}
\newtheorem{theorem}{Theorem}[section]

\newtheorem{proposition}[theorem]{Proposition}
\newtheorem{corollary}[theorem]{Corollary}
\newtheorem{lemma}[theorem]{Lemma}
\theoremstyle{definition}
\newtheorem{definition}[theorem]{Definition}
\theoremstyle{remark}
\newtheorem*{remark}{Remark}
\newtheorem*{example}{Example}
\begin{document}
\newcommand{\M}{\mathcal{M}_{g,N+1}^{(1)}}
\newcommand{\Teich}{\mathcal{T}_{g,N+1}^{(1)}}
\newcommand{\T}{\mathrm{T}}
\newcommand{\corr}{\bf}
\newcommand{\vac}{|0\rangle}
\newcommand{\Ga}{\Gamma}
\newcommand{\new}{\bf}
\newcommand{\define}{\def}
\newcommand{\redefine}{\def}
\newcommand{\Cal}[1]{\mathcal{#1}}
\renewcommand{\frak}[1]{\mathfrak{{#1}}}
\newcommand{\refE}[1]{(\ref{E:#1})}
\newcommand{\refS}[1]{Section~\ref{S:#1}}
\newcommand{\refSS}[1]{Section~\ref{SS:#1}}
\newcommand{\refT}[1]{Theorem~\ref{T:#1}}
\newcommand{\refO}[1]{Observation~\ref{O:#1}}
\newcommand{\refP}[1]{Proposition~\ref{P:#1}}
\newcommand{\refD}[1]{Definition~\ref{D:#1}}
\newcommand{\refC}[1]{Corollary~\ref{C:#1}}
\newcommand{\refL}[1]{Lemma~\ref{L:#1}}
\newcommand{\R}{\ensuremath{\mathbb{R}}}
\newcommand{\C}{\ensuremath{\mathbb{C}}}
\newcommand{\N}{\ensuremath{\mathbb{N}}}
\newcommand{\Q}{\ensuremath{\mathbb{Q}}}
\renewcommand{\P}{\ensuremath{\mathbb{P}}}
\newcommand{\Z}{\ensuremath{\mathbb{Z}}}
\newcommand{\kv}{{k^{\vee}}}
\renewcommand{\l}{\lambda}
\newcommand{\gb}{\overline{\mathfrak{g}}}
\newcommand{\g}{\mathfrak{g}}
\newcommand{\gh}{\widehat{\mathfrak{g}}}
\newcommand{\ghN}{\widehat{\mathfrak{g}_{(N)}}}
\newcommand{\gbN}{\overline{\mathfrak{g}_{(N)}}}
\newcommand{\tr}{\mathrm{tr}}
\newcommand{\sln}{\mathfrak{sl}}
\newcommand{\sn}{\mathfrak{s}}
\newcommand{\so}{\mathfrak{so}}
\newcommand{\spn}{\mathfrak{sp}}
\newcommand{\gl}{\mathfrak{gl}}
\newcommand{\slnb}{{\overline{\mathfrak{sl}}}}
\newcommand{\snb}{{\overline{\mathfrak{s}}}}
\newcommand{\sob}{{\overline{\mathfrak{so}}}}
\newcommand{\spnb}{{\overline{\mathfrak{sp}}}}
\newcommand{\glb}{{\overline{\mathfrak{gl}}}}
\newcommand{\Hwft}{\mathcal{H}_{F,\tau}}
\newcommand{\Hwftm}{\mathcal{H}_{F,\tau}^{(m)}}

\newcommand{\car}{{\mathfrak{h}}}    
\newcommand{\bor}{{\mathfrak{b}}}    
\newcommand{\nil}{{\mathfrak{n}}}    
\newcommand{\vp}{{\varphi}}
\newcommand{\bh}{\widehat{\mathfrak{b}}}  
\newcommand{\bb}{\overline{\mathfrak{b}}}  
\newcommand{\Vh}{\widehat{\mathcal V}}
\newcommand{\KZ}{Kniz\-hnik-Zamo\-lod\-chi\-kov}
\newcommand{\TUY}{Tsuchia, Ueno  and Yamada}
\newcommand{\KN} {Kri\-che\-ver-Novi\-kov}
\newcommand{\pN}{\ensuremath{(P_1,P_2,\ldots,P_N)}}
\newcommand{\xN}{\ensuremath{(\xi_1,\xi_2,\ldots,\xi_N)}}
\newcommand{\lN}{\ensuremath{(\lambda_1,\lambda_2,\ldots,\lambda_N)}}
\newcommand{\iN}{\ensuremath{1,\ldots, N}}
\newcommand{\iNf}{\ensuremath{1,\ldots, N,\infty}}

\newcommand{\tb}{\tilde \beta}
\newcommand{\tk}{\tilde \kappa}
\newcommand{\ka}{\kappa}
\renewcommand{\k}{\kappa}

\newcommand{\nw}{\nabla^{(\omega)}}

\newcommand{\Pif} {P_{\infty}}
\newcommand{\Pinf} {P_{\infty}}
\newcommand{\PN}{\ensuremath{\{P_1,P_2,\ldots,P_N\}}}
\newcommand{\PNi}{\ensuremath{\{P_1,P_2,\ldots,P_N,P_\infty\}}}
\newcommand{\Fln}[1][n]{F_{#1}^\lambda}
\newcommand{\tang}{\mathrm{T}}
\newcommand{\Kl}[1][\lambda]{\can^{#1}}
\newcommand{\A}{\mathcal{A}}
\newcommand{\U}{\mathcal{U}}
\newcommand{\V}{\mathcal{V}}
\renewcommand{\O}{\mathcal{O}}
\newcommand{\Ae}{\widehat{\mathcal{A}}}
\newcommand{\Ah}{\widehat{\mathcal{A}}}
\newcommand{\La}{\mathcal{L}}
\newcommand{\Le}{\widehat{\mathcal{L}}}
\newcommand{\Lh}{\widehat{\mathcal{L}}}
\newcommand{\eh}{\widehat{e}}
\newcommand{\Da}{\mathcal{D}}
\newcommand{\kndual}[2]{\langle #1,#2\rangle}
\newcommand{\cins}{\frac 1{2\pi\mathrm{i}}\int_{C_S}}
\newcommand{\cinsl}{\frac 1{24\pi\mathrm{i}}\int_{C_S}}
\newcommand{\cinc}[1]{\frac 1{2\pi\mathrm{i}}\int_{#1}}
\newcommand{\cintl}[1]{\frac 1{24\pi\mathrm{i}}\int_{#1 }}
\newcommand{\w}{\omega}
\newcommand{\ord}{\operatorname{ord}}
\newcommand{\res}{\operatorname{res}}
\newcommand{\nord}[1]{:\mkern-5mu{#1}\mkern-5mu:}
\newcommand{\Fn}[1][\lambda]{\mathcal{F}^{#1}}
\newcommand{\Fl}[1][\lambda]{\mathcal{F}^{#1}}
\renewcommand{\Re}{\mathrm{Re}}

\newcommand{\ha}{H^\alpha}
\newcommand{\al}{\alpha}
\newcommand{\be}{\beta}

\define\ldot{\hskip 1pt.\hskip 1pt}
\define\ifft{\qquad\text{if and only if}\qquad}
\define\a{\alpha}
\redefine\d{\delta}
\define\w{\omega}
\define\ep{\epsilon}
\redefine\b{\beta}
\redefine\t{\tau}
\redefine\i{{\,\mathrm{i}}\,}
\define\ga{\gamma}
\define\cint #1{\frac 1{2\pi\i}\int_{C_{#1}}}
\define\cintta{\frac 1{2\pi\i}\int_{C_{\tau}}}
\define\cintt{\frac 1{2\pi\i}\oint_{C}}
\define\cinttp{\frac 1{2\pi\i}\int_{C_{\tau'}}}
\define\cinto{\frac 1{2\pi\i}\int_{C_{0}}}
\define\cinttt{\frac 1{24\pi\i}\int_C}
\define\cintd{\frac 1{(2\pi \i)^2}\iint\limits_{C_{\tau}\,C_{\tau'}}}
\define\cintdr{\frac 1{(2\pi \i)^3}\int_{C_{\tau}}\int_{C_{\tau'}}
\int_{C_{\tau''}}}
\define\im{\operatorname{Im}}
\define\re{\operatorname{Re}}
\define\res{\operatorname{res}}
\redefine\deg{\operatornamewithlimits{deg}}
\define\ord{\operatorname{ord}}
\define\rank{\operatorname{rank}}
\define\fpz{\frac {d }{dz}}
\define\dzl{\,{dz}^\l}
\define\pfz#1{\frac {d#1}{dz}}

\define\K{\Cal K}
\define\U{\Cal U}
\redefine\O{\Cal O}
\define\He{\text{\rm H}^1}
\redefine\H{{\mathrm{H}}}
\define\Ho{\text{\rm H}^0}
\define\A{\Cal A}
\define\Do{\Cal D^{1}}
\define\Dh{\widehat{\mathcal{D}}^{1}}
\redefine\L{\Cal L}
\redefine\D{\Cal D^{1}}
\define\KN {Kri\-che\-ver-Novi\-kov}
\define\Pif {{P_{\infty}}}
\define\Uif {{U_{\infty}}}
\define\Uifs {{U_{\infty}^*}}
\define\KM {Kac-Moody}
\define\Fln{\Cal F^\lambda_n}
\define\gb{\overline{\mathfrak{ g}}}
\define\G{\overline{\mathfrak{ g}}}
\define\Gb{\overline{\mathfrak{ g}}}
\redefine\g{\mathfrak{ g}}
\define\Gh{\widehat{\mathfrak{ g}}}
\define\gh{\widehat{\mathfrak{ g}}}
\define\Ah{\widehat{\Cal A}}
\define\Lh{\widehat{\Cal L}}
\define\Ugh{\Cal U(\Gh)}
\define\Xh{\hat X}
\define\Tld{...}
\define\iN{i=1,\ldots,N}
\define\iNi{i=1,\ldots,N,\infty}
\define\pN{p=1,\ldots,N}
\define\pNi{p=1,\ldots,N,\infty}
\define\de{\delta}

\define\kndual#1#2{\langle #1,#2\rangle}
\define \nord #1{:\mkern-5mu{#1}\mkern-5mu:}
\define \sinf{{\widehat{\sigma}}_\infty}
\define\Wt{\widetilde{W}}
\define\St{\widetilde{S}}
\define\Wn{W^{(1)}}
\define\Wtn{\widetilde{W}^{(1)}}
\define\btn{\tilde b^{(1)}}
\define\bt{\tilde b}
\define\bn{b^{(1)}}
%
\define\eps{\varepsilon}    
\define\doint{({\frac 1{2\pi\i}})^2\oint\limits _{C_0}
       \oint\limits _{C_0}}                            
\define\noint{ {\frac 1{2\pi\i}} \oint}   
\define \fh{{\frak h}}     
\define \fg{{\frak g}}     
\define \GKN{{\Cal G}}   
\define \gaff{{\hat\frak g}}   
\define\V{\Cal V}
\define \ms{{\Cal M}_{g,N}} 
\define \mse{{\Cal M}_{g,N+1}} 
\define \tOmega{\Tilde\Omega}
\define \tw{\Tilde\omega}
\define \hw{\hat\omega}
\define \s{\sigma}
\define \car{{\frak h}}    
\define \bor{{\frak b}}    
\define \nil{{\frak n}}    
\define \vp{{\varphi}}
\define\bh{\widehat{\frak b}}  
\define\bb{\overline{\frak b}}  
\define\Vh{\widehat V}
\define\KZ{Knizhnik-Zamolodchikov}
\define\ai{{\alpha(i)}}
\define\ak{{\alpha(k)}}
\define\aj{{\alpha(j)}}
\newcommand{\laxgl}{\overline{\mathfrak{gl}}}
\newcommand{\laxsl}{\overline{\mathfrak{sl}}}
\newcommand{\laxs}{\overline{\mathfrak{s}}}
\newcommand{\laxg}{\overline{\frak g}}
\newcommand{\bgl}{\laxgl(n)}
\newcommand{\tX}{\widetilde{X}}
\newcommand{\tY}{\widetilde{Y}}
\newcommand{\tZ}{\widetilde{Z}}

\title[Multipoint Lax operator algebras]
     {Multipoint Lax operator algebras.
\\
  Almost-graded structure
\\
and central extensions}
\thanks{This work was partially supported by the IRP
GEOMQ11  of the University of Luxembourg.}
\author[M. Schlichenmaier]{Martin Schlichenmaier}
\address[Martin Schlichenmaier]
{%
University of Luxembourg\\                                           
Mathematics Research Unit, FSTC\\                                    
Campus Kirchberg\\ 6, rue Coudenhove-Kalergi,
L-1359 Luxembourg-Kirchberg\\ Luxembourg
}
\email{martin.schlichenmaier@uni.lu}
\begin{abstract}
Recently,  Lax operator algebras 
appeared  as a new  class of higher genus current type
algebras. Based on I.~Krichever's theory of Lax
operators on algebraic curves they were introduced by 
I.~Kri\-che\-ver and
O.~Sheinman. These algebras  are almost-graded Lie
algebras of currents on Riemann surfaces with marked points
(in-points, out-points, and Tyurin points). In a previous
joint article of the author with Sheinman
 the local cocycles and associated
almost-graded central extensions are classified in the case of one
in-point and one out-point. It was shown that the almost-graded
extension is essentially unique. In this article  the
general case of Lax operator algebras corresponding to several
in- and out-points is considered. 
In a first step it is shown that they are almost-graded. The grading
is given by the splitting of the marked points 
which are non-Tyurin points into in- and out-points.
Next, classification results both for local and bounded cocycles are shown. 
The uniqueness theorem for almost-graded central extensions
follows.
For this  generalization additional techniques
are needed which are presented in this article.
\end{abstract}
\date{April 11, 2013}
\maketitle
\tableofcontents
\section{Introduction}\label{S:intro}
Lax operator algebras are a recently introduced new class of
current type Lie algebras. In their full generality they were
introduced by Krichever and Sheinman in \cite{KSlax}. There the
concept of Lax operators on algebraic curves, as considered by
Krichever in \cite{Klax}, was generalized to $\g$-valued Lax operators,
where $\g$ is a classical complex Lie algebra. 
Krichever \cite{Klax} extended the conventional Lax operator
representation with a rational parameter to the case of
algebraic curves of arbitrary genus. Such generalizations of Lax
operators
appear in many fields. They are closely related to
integrable
systems (Krichever-Novikov equations on elliptic curves, elliptic
Calogero-Moser systems, Baker-Akhieser functions), see \cite{Klax},
\cite{Kiso}.
Another important application appears in the context of  moduli spaces
of bundles. In particular, they are related to 
Tyurin's result on the classification of framed semi-stable holomorphic
vector bundles on algebraic curves \cite{Tyvb}.
The classification 
uses {\it Tyurin parameters} of such
bundles, consisting of points $\ga_s$ ($s=1,\ldots ,ng$), and
associated elements $\a_s\in\P^{n-1}(C)$ (where $g$ denotes the genus
of the Riemann surface $\Sigma$, and $n$ corresponds to the rank of
the  bundle).
In the following I will not make any reference to these applications.
Beside the above mentioned work the reader might refer to 
Sheinman \cite{ShLax}, \cite{SheBo} for more background
in the case of integrable systems.

\smallskip

Here I will concentrate on the mathematical structure of these
algebras. 
Lax operator algebras are infinite dimensional Lie algebras of
geometric
origin and are interesting mathematical objects.
In contrast to the classical genus zero algebras, appearing in
Conformal Field Theory, they are not graded anymore. 
In this article we will 
introduce an almost-graded structure 
(see \refD{almgrad}) for them. Such an almost-grading
will be an indispensable tool. A crucial task for such
infinite dimensional Lie algebras is the construction and
classification of central extensions. 
This is done in the article. We will concentrate on such central
extensions for which the almost-grading can be extended.

In certain respect the Lax operator algebras can be considered as
generalizations of the higher-genus Krichever-Novikov type current 
and affine algebras, see \cite{KNFa}, \cite{Shea}, \cite{Sha}, \cite{SSS},
\cite{SHab},
\cite{Saff}. 
They themselves are generalizations of the classical
affine Lie algebras as e.g. introduced by 
Kac \cite{Kac68}, \cite{KacB} and Moody \cite{Moody69}.

This article extends the results on the two-point case (see the 
next paragraph for its definition) to the multi-point case. As far as 
the
almost-grading in the two-point case 
is concerned, see Krichever and Sheinman \cite{KSlax}. For the central
extensions in the two-point case 
see the joint work of the author with Sheinman \cite{SSlax}.

\smallskip
To describe  the obtained results we first have to give 
a rough description of the setup. Full details will be given in 
\refS{algebras}.
Let $\g$ be one of the classical Lie algebras
\footnote{As far as $G_2$ is  concerned see 
the recent preprint of Sheinman \cite{Shg2}, and the remark at the end of the
introduction.}
$\gl(n),\sln(n),\so(n),\spn(n)$ over $\C$
and $\Sigma$ a compact Riemann surface. Let $A$ be a finite set
of points of $\Sigma$ divided into two disjoint
non-empty subsets $I$ and $O$. Furthermore, let $W$
be another finite set of points (called weak singular points).
Our Lax operator algebra consists of meromorphic functions
$\ \Sigma\to\g$, holomorphic outside of $W\cup A$ with possibly  poles of
order 1 (resp. of order  2 for $\spn(n)$) at the points in $W$ and certain
additional conditions, depending on $\g$,  on the  
Laurent series expansion there (see e.g. \refE{gldef}).
It turns out \cite{KSlax} that due to the additional condition
this set of matrix-valued functions closes to a Lie algebra $\gb$ under
the point-wise commutator.
In case that $W=\emptyset$ then $\gb$ will be the 
Krichever-Novikov type current algebra (associated to this special
finite-dimensional Lie algebras).
They were extensively studied by Krichever and Novikov, Sheinman, and
Schlichenmaier
 see e.g.
\cite{KNFa}, \cite{Shea}, \cite{Sha}, \cite{ShN65},
\cite{SHab}, \cite{SSS}, \cite{Saff}.
It has to be pointed out that the Krichever-Novikov type algebras can
be defined for all finite-dimensional Lie algebras $\g$.

If furthermore, the genus of the Riemann surface is zero and $A$
consists only of two points, which we might assume to be $\{0\}$ and
$\{\infty\}$, then the algebras will be the usual classical
current algebras. These classical algebras are graded algebras. Such
a grading is used e.g. to introduce highest weight representations,
Verma modules, Fock spaces,  and to classify these representations. 
Unfortunately, the algebras which we consider here will not be
graded. But they admit an almost-grading, see \refD{almgrad}.
As was realized by Krichever and Novikov \cite{KNFa}
for most applications it is a valuable replacement for the 
grading.
They also gave a method how to introduce it for the two-point algebras
of Krichever-Novikov type.

For the multi-point case of the Krichever-Novikov type algebras such
an almost-grading was given by the author
\cite{SDiss}, \cite{SLc}, \cite{Saff}, \cite{Schknrev},
\cite{Schlknbook},
see also \cite{Sad}.
The crucial point is that the almost-grading will depend on the
splitting
of $A$ into $I$ and $O$. Different splittings will give
different almost-gradings. Hence, the multi-point
case is more involved than he two-point case. 

For the Krichever-Novikov current algebra $\gb$ the grading
comes from the grading of the function algebra (to be found in the
above cited works of the author). 
This is due to the fact, that they are tensor products.
If $W\ne\emptyset$ the Lax operator algebras are not tensor products
anymore and their almost-grading has to be constructed
directly. This has been done in the two point case by Krichever and
Sheinman
\cite{KSlax}.

\smallskip
Our first result in this article is to introduce an almost-grading
of $\gb$
for the arbitrary multi-point case. As mentioned above, it will depend
in an essential way on the splitting of $A$ into $I\cup O$. 
This is done in \refS{alm}.
The construction is much more involved than in the two-point case.

\smallskip
Our second goal is to study central extensions $\gh$ of the Lax
operator 
algebras $\gb$. It is well-known that central extensions are given by
Lie algebra two-cocycles of $\gb$ with values in the trivial module
$\C$.
Equivalence classes of central extensions are in 1:1 correspondence to the 
elements of the Lie algebra cohomology space $\H^2(\gb,\C)$. Whereas
for the classical current algebras associated to a finite-dimensional
simple Lie algebra $\g$ the extension class will be unique 
this is not the case anymore for higher genus
and even for genus zero in the multi-point case.
But we are interested only in central extensions $\gh$ which allow
us to extend the almost-grading of $\gb$.
This reduces the possibilities. The condition for the
cocycle defining the central extension
will be that it is  {\it local} (see \refE{cloc}) with
respect to the almost-grading given by the splitting $A=I\cup O$.
Hence, which cocycles will be local will depend on the splitting
as well.

If $\g$ is simple then the space of local cohomology classes for $\gb$
will be one-dimensional.
For $\gl(n)$ we have to add another natural property for the
cocycle meaning that it is invariant under the action of 
the vector field algebra $\La$ (see \refE{linv}). 
In this case the space of local and $\L$-invariant cocycle classes
will be two-dimensional.

The action of the vector field algebra $\La$ on $\gb$
is given 
in terms  of a certain  connection
$\nabla^{(\omega)}$, see \refS{Lmod}. 
With the help of the connection we can define
geometric cocycles 
\begin{align}
\ga_{1,\w,C}(L,L')&= \cinc {C} \tr(L\cdot \nabla^{(\w)}L'),
\\
\label{E:ig2} \ga_{2,\w,C}(L,L')&= \cinc{C} \tr(L)\cdot
\tr(\nabla^{(\w)}L'),
\end{align}
where $C$ is an arbitrary cycle on
$\Sigma$ avoiding the points of possible singularities.
The cocycle  $\ga_{2,\w,C}$ will only be different from zero 
in the $\gl(n)$ case.

Special integration paths are circles $C_i$ around the points
in $I$, resp. around the points in $O$, and a path $C_S$ separating the
points in $I$ form the points in $O$.

Our main result is \refT{main} about uniqueness of local cocycles classes
and that the cocycles are given by integrating along $C_S$.
The proof presented in \refS{class} is based on \refT{bounded} which 
gives the 
classification of bounded (from above) cohomology classes (see 
\refE{bounded}).
The bounded cohomology classes constitute a subspace of dimension $N$, 
(resp. $2N$ for $\gl(n)$) where $N=\#I$ and the integration is
done over the $C_i, i=1,\ldots, N$.

The proof of \refT{bounded} is given in \refS{induction} and \refS{direct}. 
We use recursive techniques as developed in
\cite{Scocyc} and \cite{Saff}.
Using  the boundedness and        
$\L$-invariance we  show that such a  cocycle is given by its values  
at  pairs of homogeneous elements for which the sum of their    
degrees is equal to zero. Furthermore, we show that an            
$\L$-invariant and bounded cocycle will be uniquely fixed by a      
certain finite number of such cocycle values. A more detailed     
analysis shows that the cocycles are of the form claimed.
In  \refS{direct} we show  that  in the
simple Lie algebra case in each bounded  cohomology class there is a 
representing cocycle which is            
$\L$-invariant. 
For this we  use the internal structure of the  Lie algebra $\gb$ related
to the root system of the underlying finite dimensional      
simple Lie algebra $\g$, and                                 
the almost-gradedness of $\gb$. Recall that in the classical case
$\g\otimes\C[z,z^{-1}]$ the algebra is graded. In this very      
special case the  chain of arguments gets simpler and is similar 
to the arguments of  Garland \cite{Gar}.

\medskip

As already mentioned above, in joint work with Sheinman 
\cite{SSlax} the two-point case was considered.
This article extends the result to the multi-point
case. Unfortunately,
it is not an  application of the results of the two-point
case.
In this more general context 
the proofs have to be done anew.
(The two-point case will finally be a special case.)
Only at few places  references to proofs in
\cite{SSlax} can be made.

\medskip
I like to thank Oleg Sheinman for 
extensive discussions which were very helpful during writing
this article.  
After I finished this work he succeeded \cite{Shg2} to give a definition of
a Lax operator algebra for the exceptional Lie algebra $G_2$
in such a way, that all properties and statements presented here 
will also be true in this case.
Hence, there is now another element in the list of
Lax operator algebra associated to simple Lie algebras.
\section{The algebras} 
\label{S:algebras}
\subsection{Lax operator algebras}\label{S:alg}
$ $

Let $\g$ be  one of the classical matrix
algebras $\gl(n)$, $\sln(n)$, $\so(n)$, $\spn(2n)$, or $\sn(n)$,
where the latter  denotes the algebra of scalar matrices.
Our algebras  will consist of certain $\g$-valued meromorphic  
functions, forms, etc, defined on Riemann surfaces with additional
structures (marked points, vectors associated to this points, ...).

To become more precise, 
let $\Sigma$ be a compact Riemann surface of genus $g$
($g$ arbitrary)
and $A$ a finite subset of points in $\Sigma$ divided into
two non-empty disjoint subsets
\begin{equation}
I:=\{P_1,P_2,\ldots,P_N\}, \qquad O:=\{Q_1,Q_2,\ldots,Q_{M}\}
\end{equation}
with $\#A=N+M$.
The points in $I$ are called incoming-points the points in $O$ 
outgoing-points.

To define 
Lax operator algebras we have to fix some
additional data.
Fix $K\in\N_0$ and a collection of points 
\begin{equation}
W:=\{\ga_s\in\Sigma\setminus A\mid s=1,\ldots, K\}.
\end{equation}
We assign to every point $\ga_s$ 
 a vector $\al_s\in\C^n$ 
(resp. from $\C^{2n}$ for  $\spn(2n)$).
The system 
\begin{equation}
\mathcal{T}:=\{(\ga_s,\al_s)\in\Sigma\times \C^n\mid s=1,\ldots, K\}
\end{equation}
is called 
{\it Tyurin data}. 
We will be more general 
than in our earlier joint paper \cite{SSlax} with 
Sheinman, not only in respect
that we allow for $A$ more than two points also that
our $K$ is not bound to be $n\cdot g$.
Even $K=0$ is allowed. In the latter case the Tyurin data will be
empty.
\begin{remark}
For $K=n\cdot g$ 
and for  generic values of $(\ga_s,\a_s)$ with $\a_s\ne 0$ the tuples
of pairs
 $(\ga_s,[\a_s])$ with $[\a_s]\in\P^{n-1}(\C)$
parameterize framed semi-stable rank $n$ and degree $n\, g$ holomorphic
vector bundles as shown by Tyurin \cite{Tyvb}.
Hence, the name Tyurin data.
\end{remark}
We fix local coordinates 
$z_l,\ l=1,\dots, N$  centered at the points $P_l\in I$ 
and $w_s$ centered at $\ga_s$,  $s=1,\ldots, K$.
In fact nothing  will dependent on the choice of 
 $w_s$. This is essentially also true for $z_l$. 
Only its
first jet will be used to normalize certain
basis elements uniquely.

We consider $\g$-valued meromorphic functions
\begin{equation}
L:\ \Sigma\ \to\  \g,
\end{equation}
which are holomorphic outside  $W\cup A$, have at most
poles of order one (resp. of order two for $\spn(2n)$) at the
points in $W$, and fulfill certain conditions at $W$ depending on
$\mathcal{T}$, $A$,  and $\g$. These conditions will be 
described in the following. The singularities at $W$ are called {\it weak
singularities}. These objects were introduced by Krichever
\cite{Klax}  for $\gl(n)$ in the context of Lax operators for
algebraic curves, and further generalized 
by Krichever and Sheinman in \cite{KSlax}.
The conditions are exactly the same as in \cite{SSlax}. But for 
the convenience of the reader we recall them here.

\medskip
\noindent
For {\bf $\gl(n)$} the conditions 
are as follows. 
For $s=1,\ldots, K$ we require that there exist $\be_s\in\C^n$ 
and $\ka_s\in \C$ such that the
function $L$ has the following expansion at $\ga_s\in W$ 
\begin{equation}\label{E:glexp}
L(w_s)=\frac {L_{s,-1}}{w_s}+
L_{s,0}+\sum_{k>0}L_{s,k}{w_s^k},
\end{equation}
with
\begin{equation}\label{E:gldef}
L_{s,-1}=\al_s \be_s^{t},\quad
\tr(L_{s,-1})=\be_s^t \al_s=0,
\quad
L_{s,0}\,\al_s=\ka_s\al_s.
\end{equation}
In particular, if $L_{s,-1}$ is 
non-vanishing then it is a rank 1 matrix, and if 
$\al_s\ne 0$  then it is  
an eigenvector of $L_{s,0}$.

The requirements \refE{gldef} are independent of the chosen
coordinates $w_s$ and  
the set of all such functions constitute an associative algebra under
the point-wise matrix multiplication, see \cite{KSlax}. 
The proof transfers without changes to the multi-point case.
For the convenience of the reader 
and for illustration we will  nevertheless 
recall the
proof in an appendix to this article.
We denote this algebra  by $\glb(n)$.
Of course, it will depend on the Riemann surface $\Sigma$, 
the finite set of points $A$, and the Tyurin data $\T$.
As there should be no confusion, we prefer to avoid
cumbersome notation and will just use  $\glb(n)$.
The same we do for the other Lie algebras.

Note that if one of the $\alpha_s=0$ then 
the conditions at the point $\gamma_s$ correspond to the fact, 
that $L$ has to be
holomorphic there. We can erase the point from the
Tyurin data. Also if $\alpha_s\ne 0$ and $\lambda\in\C, \lambda\ne 0$
then $\alpha$ and $\lambda\alpha$ induce the same conditions
at the point $\gamma_s$. Hence only the projective vector
$[\alpha_s]\in\P^{n-1}(\C)$ plays a role.

The splitting $\gl(n)=\sn(n)\oplus \sln(n)$ given by
\begin{equation}
X\mapsto \left(\ \frac {\tr(X)}{n}I_n\ ,\ X-\frac {\tr(X)}{n}I_n\ \right),
\end{equation}
where $I_n$ is the $n\times n$-unit matrix,
induces a corresponding splitting for the  Lax operator 
algebra  $\glb(n)$:
\begin{equation}\label{E:glsplit}
 \glb(n)=\snb(n)
\oplus \slnb(n).
\end{equation}
For {\bf $\slnb(n)$} the only 
additional condition  is that 
in \refE{glexp} all matrices $L_{s,k}$ are  trace-less.
The conditions \refE{gldef}  remain unchanged.

For {\bf $\snb(n)$} all matrices in \refE{glexp}
are scalar  matrices. This implies that
the corresponding 
$L_{s,-1}$  vanish. In particular, the elements
of $\snb(n)$ are holomorphic at $W$.
Also $L_{s,0}$, as a scalar matrix, has every $\al_s$ 
as eigenvector.
This means that beside the holomorphicity there are no
further conditions.
And we get
$\snb(n)\cong \A$,
where  $\A$ be the (associative) algebra of meromorphic
functions on $\Sigma$ holomorphic outside of $A$.
This is 
the (multi-point) Krichever-Novikov type function algebra.
It will be discussed further down in  \refS{kngrad}.

\medskip

In the case of {\bf $\so(n)$} we require that
all $L_{s,k}$ in \refE{glexp} are  skew-symmetric.
In particular, they are trace-less.
Following \cite{KSlax} the set-up has to be slightly modified.
First only  those Tyurin parameters $\al_s$ are allowed which satisfy
$\al_s^t\al_s=0$.
Then the first requirement in  \refE{gldef} is changed 
to obtain
\begin{equation}\label{E:sodef}
L_{s,-1}=\al_s\be_s^t-\be_s\al_s^t,
\quad
\tr(L_{s,-1})=\be_s^t\al_s=0,
\quad
L_{s,0}\,\al_s=\ka_s\al_s.
\end{equation}

\medskip
For {\bf $\spn(2n)$}
we consider  a symplectic form  $\hat\sigma$  
for $\C^{2n}$ given by
a non-degenerate skew-symmetric matrix $\sigma$.
The Lie algebra $\spn(2n)$ is the Lie algebra of
matrices $X$ such that $X^t\sigma+\sigma X=0$.
The condition $\tr(X)=0$ will be automatic. 
At the weak singularities we have the expansion
\begin{equation}\label{E:glexpsp}
L(z_s)=\frac {L_{s,-2}}{w_s^2}+\frac {L_{s,-1}}{w_s}+
L_{s,0}+L_{s,1}{w_s}+\sum_{k>1}L_{s,k}{w_s^k}.
\end{equation}
The condition \refE{gldef} is  modified as
follows (see \cite{KSlax}):
there exist $\be_s\in\C^{2n}$, 
$\nu_s,\ka_s\in\C$ such that 
\begin{equation}\label{E:spdef}
L_{s,-2}=\nu_s \al_s\al_s^t\sigma,\quad
L_{s,-1}=(\al_s\be_s^t+\be_s\al_s^t)\sigma,
\quad{\be_s}^t\sigma\al_s=0,\quad
L_{s,0}\,\al_s=\kappa_s\al_s.
\end{equation}
{}Moreover, we require
\begin{equation}
\al_s^t\sigma L_{s,1}\al_s=0.
\end{equation}
Again under the point-wise
matrix commutator the set of such maps constitute a Lie algebra.
\begin{theorem}
Let $\gb$ be the 
space 
of Lax operators  associated to $\g$, one of the  
above introduced finite-dimensional
classical Lie algebras. Then  $\gb$ is a Lie 
algebra under the
point-wise matrix commutator. 
For $\gb=\glb(n)$ 
it is an associative 
algebra under point-wise matrix multiplication.
\end{theorem}
The proof in \cite{KSlax}
extends without problems to 
the multi-point situation (see the appendix for an 
example).

These Lie algebras are called {\it Lax operator algebras}.

\subsection{Krichever-Novikov algebras of current type}
\label{S:kncurr}
$ $

Let  $\A$ be the (associative) algebra of meromorphic
functions on $\Sigma$ holomorphic outside of $A$.
Let $\g$ be an arbitrary finite-dimensional Lie algebra.
On the tensor product $\g\otimes\A$ a Lie algebra structure is
given by
\begin{equation}\label{E:mother}
[x\otimes f, y\otimes g]:=[x,y]\otimes (f\cdot g),
\quad x,y\in\g,\  f,g\in \A.
\end{equation}
The elements of this 
Lie algebra can be considered as the set of those
meromorphic maps
$\ \Sigma\to\g $,
which are holomorphic outside of $A$.
These algebras are called \emph{(multi-point) Krichever Novikov algebras
of current type}, see
\cite{KNFa}, \cite{KNFb}, \cite{KNFc}, 
\cite{Shea}, \cite{Sha}, \cite{SHab}, \cite{Saff}.

If the genus of the surface is zero and if $A$ consists
of two points, the Krichever-Novikov current algebras 
are the
classical current (or loop algebra)
$\g\otimes\C[z^{-1},z]$.

In the case that in
the defining data  of the Lax operator algebra
there are no weak singularities, resp. all $\alpha_s=0$,
then for the $\g$-valued 
meromorphic functions the requirements reduce to the 
condition  that they are holomorphic
outside of $A$. Hence, we obtain (for these $\g$)  the
Krichever-Novikov
current type algebra.
But note that not for all finite-dimensional $\g$ we have an
extension of the notion Krichever-Novikov current 
to a Lax operator algebra.
 
\section{The almost-graded structure}
\label{S:alm} 
\subsection{The statements}
\label{S:almstruct}
$ $

For the construction of certain important representations
of infinite dimensional Lie algebras 
(Fock space representations, Verma modules, etc.)
a graded structure is usually assumed and heavily used.
The algebras we are considering for higher genus, or
even for genus zero with many marked points were poles 
are allowed, cannot be nontrivially graded.
As realized by Krichever and Novikov \cite{KNFa}
a weaker concept, an almost-grading, will
be enough to allow to do the above mentioned constructions.
\begin{definition}\label{D:almgrad}
A Lie algebra $V$ will be called {\it almost-graded} (over $\Z$) 
if there exists finite-dimensional subspaces $V_m$ and constants
$S_1,S_2\in\Z$ such that
\begin{enumerate}
\item \quad
$V=\bigoplus_{m\in\Z}V_m$,
\item\quad
$\dim V_m<\infty$,\quad $\forall m\in\Z$,
\item\quad
$[V_n,V_m]\quad \subseteq\quad \sum_{h=n+m+S_1}^{n+m+S_2} V_h$.
\end{enumerate}
If there exists an $R$ such that $\dim V_m\le R$ for all $m$ it is
called \emph{strongly  almost-graded}.
\end{definition}
Accordingly, an almost-grading can be defined for associative algebras
and for modules over almost-graded algebras.

We will introduce for our 
multi-point Lax operator in the following such a
(strong) almost-graded structure. The almost-grading will be induced by the
splitting of our set $A$ into $I$ and $O$.
Recall that $I=\{P_1,P_2,\ldots,P_N\}$.
In the Krichever Novikov function, vector field, and current algebra
case this was done  by Krichever and Novikov
\cite{KNFa}
for the two-point situation.
In the two-point Lax operator algebra it was done by Krichever and
Sheinman 
\cite{KSlax}. In the two-point case there is only one 
splitting possible.
This is in contrast to the multi-point case which turns out
to be more difficult. 
The multi-point Krichever-Novikov algebras of different types
were done by Schlichenmaier
\cite{SDiss},\cite{SLc}.
We will recall it in \refS{kngrad}.

In \refS{almproof}  we will single out for each $m\in\Z$
a subspace $\gb_m$ of $\gb$,
called {\it (quasi-)homogeneous} subspace of degree $m$.
The degree is essentially related to  the order of the 
elements of $\gb$ at the points in $I$.
We will show
\begin{theorem}\label{T:almgrad}
Induced by  the splitting $A=I\cup O$  
the (multi-point) Lax operator algebra $\gb$
becomes a (strongly)  almost-graded Lie algebra
\begin{equation}\label{E:almgrad}
\begin{gathered}
\gb=\bigoplus_{m\in\Z}\gb_m, \qquad \dim\gb_m=N\cdot\dim\g
\\
[\gb_m,\gb_n]\quad\subseteq\quad\bigoplus_{h=n+m}^{n+m+S}\gb_h,
\end{gathered}
\end{equation}
with a constant $S$ independent of $n$ and $m$.
\end{theorem}
In addition we will show
\begin{proposition}\label{P:locex}
Let $X$ be an element of $\g$. For each $(m,s)$, $m\in\Z$ and
$s=1,\ldots,N$
  there is
a unique element $X_{m,s}$ in $\gb_m$ such that 
locally in the neighbourhood of the point $P_p\in I$ we have
\begin{equation}\label{E:locex}
{X_{m,s}}_{|}(z_p)=
Xz_p^m\cdot\delta_s^p+O(z_p^{m+1}),\quad \forall p=1,\ldots,N.
\footnote{The symbol $\delta_s^p$ denotes the Kronecker delta, which is
equal to 1 if $s=p$, otherwise $0$.}
\end{equation}
\end{proposition}
\begin{proposition}\label{P:basis}
Let $\{X^u\mid u=1,\ldots,\dim\g\}$ 
be a basis of the finite dimensional
Lie algebra $\g$. Then  
\begin{equation}\label{E:set}
\mathcal{B}_m:=
\{X^u_{m,p},\  u=1,\ldots, \dim\g,\  p=1,\ldots, N \}
\end{equation}
is a basis 
of $\gb_m$, and $\mathcal{B}=\cup_{m\in\Z}\;\mathcal{B}_m$
is a basis of $\gb$.
\end{proposition}
\begin{proof}
By \refE{almgrad} we know that $\dim\gb_m=N\cdot \dim\g$.
The elements in $\mathcal{B}_m$ are pairwise different. Hence,
we have $\#\mathcal{B}_m=N\cdot \dim\g$
elements $\{X^u_{m,p}\}$ in $\gb_m$.
For being a basis it suffices to show that they are linearly 
independent.
Take $\sum_{u}\sum_p\a_{m,p}^uX_{m,p}^u=0$ a linear combination
of zero. We consider the local  expansions at the point
$P_s$, for $s=1,\ldots, N$. From  \refE{locex}
we obtain
\begin{equation*}
0=
(\sum_{u}\a_{m,s}^uX^u)z_s^m+O(z_s^{m+1}).
\end{equation*}
Hence $0=\sum_{u}\a_{m,s}^uX^u$. As the $X^u$ are a basis of $\g$ this
implies
that $a_{m,s}^u=0$ for all $u,s$.
That $\mathcal{B}$ is a basis of the full $\gb$ 
follows from the direct sum
decomposition
in \refE{almgrad}.
\end{proof}

It is very convenient to introduce the associated filtration
\begin{equation}\label{E:filt}
\gb_{(k)}:=\bigoplus_{m\ge k} \gb_m,\qquad
\gb_{(k)}\subseteq \gb_{(k')},\quad  k\ge k'.
\end{equation}
\begin{proposition}\label{P:filt}
$ $

\smallskip\noindent
(a) \quad $\gb=\bigcup_{m\in\Z}\gb_{(m)}$,

\smallskip\noindent
(b) \quad $[\gb_{(k)},\gb_{(m)}]\subseteq  \gb_{(k+m)}$,

\smallskip\noindent
(c)  \quad $\gb_{(m)}/\gb_{(m+1)}\cong \gb_m$.

\smallskip\noindent
(d) The equivalence classes of the 
elements of the set $\mathcal{B}_m$ (see \refE{set})
constitute a basis 
for the quotient space $\gb_{(m)}/\gb_{(m+1)}$.
\end{proposition}  
\begin{proof}
Equation \refE{almgrad} implies directly 
(a), (b), and (c). Part (d)  follows from \refP{basis}.
\end{proof}
There is another filtration.
\begin{equation}\label{E:filt1}
\gb'_{(m)}:=\{L\in\gb\mid \ord_{P_s}(L)\ge m,\ s=1,\ldots, N\}.
\end{equation}
Note that the elements
$L$ are meromorphic maps from $\Sigma$ to $\g$, hence it makes sense
to talk about the orders of the component functions with
respect to a basis. 
The minimum of these orders is meant in \refE{filt1}.
\begin{proposition}\label{P:filt2}
$ $

\smallskip\noindent
(a) \quad $\gb=\bigcup_{m\in\Z}\gb'_{(m)}$.

\smallskip\noindent
(b) The two filtrations  coincide, i.e.
\begin{equation*}
\gb_{(m)}=\gb'_{(m)},\qquad\forall m\in\Z.
\end{equation*}
\end{proposition}
\begin{proof}
Let $L\in\gb$, then as $\g$-valued meromorphic functions the 
pole orders of the component functions at the points $P_s$
are individually bounded. As there are only finitely many, there
is a bound $k$ for the pole order, hence $L\in\gb_{(-k)}$. This
shows (a) and consequently that $(\gb'_{(m)})$ is a filtration.
\newline
By \refP{basis} we know that $\mathcal{B}$ is a basis of $\gb$.
Let $L\in\gb'_{(m)}$.
Every element of $L\in\gb$ will be a finite linear combination
 of the basis elements. 
The elements of $\mathcal{B}_k$ have 
exact order $k$  
and are linearly independent.
Moreover, with respect to a fixed basis element of the finite
dimensional Lie algebra we have $N$ basis elements in $\mathcal{B}_k$ with
orders given by \refE{locex}. Hence the individual orders
at the points $P_s$  cannot increase
with non-trivial linear combinations. Hence 
 only $k\ge m$ can appear
in
the combination. This shows $L\in\gb_{(m)}$.
 Vice versa, obviously
all elements from $\mathcal{B}_k$ for $k\ge m$ lie in the
set \refE{filt}. Hence, we have equality.
\end{proof}
The second description of the filtration has the big advantage, 
that it is very naturally defined. The only data which enters is the
splitting of the points $A$ into $I\cup O$.
Hence, it is canonically given by $I$.
In contrast, it will turn out that in the multi-point case 
if $\#O>1$ there might
be some choices necessary to fix $\gb_m$, like numbering the points in $O$, resp.
even some different rules for the points in $O$. But via 
\refP{filt2} we know that the induced filtration \refE{filt} will
not depend on any of these choices.

Here we have to remark that we supplied above a  proof of
\refP{filt2}.
But it  
was based on results (i.e. \refT{almgrad} and \refP{locex})
which we only  will prove in 
\refS{almproof}. Our starting point there will be 
the filtration $\gb'_{(m)}$, hence we cannot 
 assume equality
from the
very beginning.

\medskip
We have the very important fact 
\begin{proposition}\label{P:almprod}
Let $X_{k,s}$ and $Y_{m,p}$ be the elements in $\gb_k$ and $\gb_m$ 
corresponding to $X,Y\in\g$ respectively then 
\begin{equation}\label{E:alalg1}
[X_{k,s},Y_{m,p}]={[X,Y]}_{k+m,s}\delta_s^p+L, 
\end{equation}
with $[X,Y]$ the bracket in $\g$ and $L\in \gb_{(k+m+1)}$.
\end{proposition}
\begin{proof}
Using for $X_{k,s}$ and $Y_{m,p}$ the expression \refE{locex} we obtain
$$
{[X_{k,s},Y_{m,p}]}_|(z_t)=[X,Y]z_s^{k+m}\delta_t^p\delta_t^s+O(z_t^{k+m+1}), 
$$
for  every  $t$.
Hence, the element
$$
[X_{k,s},Y_{m,p}]-([X,Y])_{k+m,s}\delta_s^p
$$
has at all points in $I$ an order $\ge k+m+1$. With \refE{filt1} and 
\refP{filt2}
we obtain that it lies in $\gb_{(k+m+1)}$, 
which is the claim.
\end{proof}
\subsection{The function algebra 
 $\A$ and the vector field algebra  $\L$}\label{S:kngrad}
$ $

Before we supply the proofs of the statements in \refS{almstruct}
we want to introduce those Krichever-Novikov type algebras which
are of relevance in the following. We start with the 
Krichever-Novikov function algebra $\A$ and the
Krichever-Novikov vector field algebra $\L$. Both algebras are 
almost-graded algebras 
\begin{equation}
\A=\bigoplus_{m\in\Z} \A_m,\qquad
\L=\bigoplus_{m\in\Z} \L_m,
\end{equation}
where the almost-grading is 
induced by the same splitting of $A$ into $I\cup O$ as used for
defining the Lax operator algebras.
Recall that $I=\{P_1,\ldots,P_N\}$ and $O=\{Q_1,\ldots, Q_M\}$.

Let $\A$, respectively
$\L$, be the space of meromorphic functions, respectively of
meromorphic vector fields on $\Sigma$, holomorphic on
$\Sigma\setminus A$. In particular, they are holomorphic
also at the points in $W$. Obviously, $\A$ is an associative
algebra under the product of functions and $\L$ is a Lie algebra
under the Lie bracket of vector fields. 
In the two point case their almost-graded structure was introduced 
by Krichever and Novikov \cite{KNFa}. In the multi-point case they
were given by Schlichenmaier \cite{SLc}, \cite{SDiss}.
The results will be described in the following.

The homogeneous spaces $\A_m$ have as basis 
the set of functions
$\{A_{m,s},s=1,\ldots, N\}$
given by the conditions
\begin{equation}\label{E:Aset}
\ord_{P_i}(A_{m,s})
=(n+1)-\delta_{i}^s,\quad i=1,\ldots, N,
\end{equation}
and certain compensating conditions at the points in $O$ to 
make it unique up to multiplication with a scalar. 
For example, in case that $\#O=M=1$ and the genus is either 0, or $\ge
2$, and the points are in generic position, then the condition is
(with the exception for finitely many $m$)
\begin{equation}
\ord_{Q_M}(A_{m,s})=-N\cdot(n+1)-g+1.
\end{equation}
To make it unique we require for the local expansion at
the $P_s$  (with respect to the chosen local coordinate $z_s$)
\begin{equation}\label{E:Anorm}
{A_{n,s}}_|(z_s)=z_s^n+O(z_s^{n+1}).
\end{equation}

For the  vector field algebra $\L_m$ 
we have the  basis  $\{e_{m,s}\mid s=1,\ldots, N\}$, 
where the elements $e_{m,s}$ 
are given by the condition
\begin{equation}
\ord_{P_i}(e_{m,s})
=(n+2)-\delta_{i}^s,\quad i=1,\ldots, N,
\end{equation}
and corresponding compensating conditions at the points in $O$ to 
make it unique up to multiplication with a scalar. 
In exactly the same special situation as above 
the condition is
\begin{equation}
\ord_{Q_M}(e_{m,s})=-N\cdot(n+2)-3(g-1).
\end{equation}
The local expansion at
$P_s$ is
\begin{equation}
{e_{n,s}}_|(z_s)=(z_s^{n+1}+O(z_s^{n+2}))\frac d{dz_s}.
\end{equation}
There are constants $S_1$ and $S_2$ 
(not depending on $m,n$) such that
\begin{equation}
\A_k\cdot \A_m\subseteq
\bigoplus_{h=k+m}^{k+m+S_1}\A_h,
\qquad
[\L_k,\L_m]\subseteq
\bigoplus_{h=k+m}^{k+m+S_2}\L_h.
\end{equation}
This says that we have almost-gradedness.
In what follows we will need the  fine structure of
the almost-grading
\begin{align}
A_{k,s}\cdot A_{m,t}&=A_{k+m,s}\,\delta_s^t+Y, \qquad Y\in
\sum_{h=k+m+1}^{k+m+S_1}
\A_h,
\\
[e_{k,s},e_{m,t}]&=(m-k)\,e_{k+m,s}\,
\delta_s^t+Z,\qquad Z\in \sum_{h=k+m+1}^{k+m+S_2}
\L_h.
\end{align}
Again we have the induced filtrations $\A_{(m)}$
and $\L_{(m)}$.

\medskip

The elements of the Lie algebra $\L$ act on $\A$ as derivations.
This makes the space $\A$  an almost-graded module over $\L$. In
particular, we have
\begin{equation}
e_{k,s}\ldot A_{m,r}=mA_{k+m}\,\delta_s^r+U,
\qquad
 U\in \sum_{h=k+m+1}^{k+m+S_3}\A_h,
\end{equation}
with a constant $S_3$ not depending on $k$ and $m$. 

\medskip

Induced by the almost-grading of $\A=\oplus_{m}\A_m$ we get an
almost-grading for the Krichever-Novikov type algebra of
current type by setting
\begin{equation}
\g\otimes\A=\bigoplus_{m\in\Z}(\g\otimes\A)_{m} \quad \text{with}\quad
(\g\otimes\A)_{m}:=\g\otimes\A_{m},\ \forall m\in\Z.
\end{equation}

\subsection{The proofs}
\label{S:almproof}
$ $

Readers being in a hurry, or readers only interested in the results
may skip this rather technical section 
(involving Riemann-Roch type arguments) 
during a first reading and jump directly to 
\refS{mod}.

Recall the definition 
\begin{equation}
\gb'_{(m)}:=\{L\in\gb\mid \ord_{P_s}(L)\ge m,\ s=1,\ldots, N\}
\end{equation}
of the filtration. We will only deal 
with this filtration in this section, hence for
notational reason we will drop the ${}'$ in the following.
Finally, the primed and unprimed will coincide.
\begin{proposition}\label{P:filtbas}
Given $X\in\g$, $X\ne 0$, $s=1,\ldots,N$, $m\in\Z$  then there exists 
at least one $X_{m,s}$ such that
\begin{equation}\label{E:pal1}
{X_{m,s}}_|(z_p)=Xz_s^m\,\delta_p^s+O(z_p^{m+1}).
\end{equation}
\end{proposition}
The proof is based on the theorem of Riemann-Roch.
The technique will be used all-over in this section. Hence, we will
introduce some notation, before we proceed with the proof.
For any $m\in\Z$ we will consider certain divisors
\begin{equation}\label{E:div_m}
 D_m=(D_m)_I+D_W+(D_m)_O.
\end{equation}
Where 
\begin{equation}
\begin{aligned}
(D_m)_I&=-m\sum_{s=1}^NP_s,
\\
(D_m)_O&=\sum_{s=1}^Ma_{s,m}Q_s, \quad a_{s,m}\in\Z
\\
D_W&=\epsilon\sum_{s=1}^K\ga_s,
\quad \epsilon=1,\text{ for } \gl(n),\sln(n),\so(n),
\quad \epsilon=2,\text{ for } \spnb(n).
\end{aligned}
\end{equation}
Recall that the genus of $\Sigma$ is $g$. Denote by $\mathcal{K}$ 
a canonical divisor. Set $L(D)$ the space  
consisting of meromorphic
functions $u$ on $\Sigma$ 
for which we have for their divisors $(u)\ge-D$.
Riemann-Roch says
\begin{equation}
\dim L(D)-\dim L(\mathcal{K}-D)=\deg D-g +1.
\end{equation}
In particular, we have
\begin{equation}\label{E:rab}
\dim L(D)\ge \deg D -g+1.
\end{equation}
We have several cases which we will need in the following
\begin{enumerate}
\item
If  $\deg D\ge 2g-1$ then we have equality in
\refE{rab}.
\item
If $D$ is a generic divisor then also for  
$g\le \deg D\le 2g-2$ we have equality.
\item
If $D\ge 0$ and $D$ is generic we have
$\dim L(D)=1$ for $0\le \deg D\le g-1$.
\item
If $D\not\ge 0$ (meaning that there is at least one
point in the support of $D$ with negative multiplicity) 
and $D$ is generic we have
$\dim L(D)=0$ for $0\le \deg D\le g-1$.
\item
For $g=0$ every divisor is generic and we have equality in 
\refE{rab} as long as the right hand side is $\ge 0$, i.e.
$\dim L(D)=\max(0,\deg D+1)$.
\end{enumerate}
See e.g. \cite{SchlRS} for
informations on divisors,  Riemann-Roch and 
their applications, see also \cite{Grif}.

In  case that $u=(u_1,u_2,\ldots, u_{r})$ 
is a vector valued function we define $L(D)$
to be the vector space of vector valued functions 
with $(u)\ge-D$. This means that
$(u_i)\ge-D$ for all $i=1,\ldots,r$.
Now all dimension formulas have to be multiplied by
$r$:
\begin{equation}\label{E:rrab}
\dim L(D)\ge r(\deg D -g+1).
\end{equation}
We apply this to our Lax operator algebra $\gb$ by considering
the component functions $u_i$, $i=1,\ldots,r=\dim\g$
with respect to a fixed basis.
We set
\begin{equation}
L'(D):=\{u\in L(D)\mid u\ \text{gives an element of } \gb\}
\subseteq L(D).
\end{equation}
In $L'(D_m)$ we have to take into account that at the weak singular points 
$\ga_s$ we
have $H$ additional linear conditions for the elements of
the solution space $L(D_m)$
to be fulfilled. They are formulated   
in terms of the corresponding $\al_s$
for some finite part of the Laurent series.
In total this are finitely many conditions.
In case that the $\al_s$ are generic they will exactly compensate
for the possible poles at $\ga_s$
\cite{KSlax}. But for the moment  we 
still consider them to be arbitrary.

\medskip

By the very definition of the filtration we always have
\begin{equation}
\gb_{(m)} = L'((D_m)_I)
\quad\text{and}\quad
\gb_{(m)} \ge L'(D_m).
\end{equation}

\medskip

\begin{proof} (\refP{filtbas})
We start with a divisor $D_m$ by choosing the part $(D_{m})_O=T$ such 
that  
the degree of the divisors $D_m$ and $D_m-\sum P_i$ is still big
enough such
that for both the case (1) of the Riemann-Roch equality 
\refE{rrab} is true and that $\dim L(D)=l\ge r(N+1)+H$. 
Hence, after applying the $H$ linear conditions we  have
$\dim L'(D)\ge r(N+1)$. 
Let $P_s$ be a fixed point from $I$.
We consider 
\begin{equation}
D_m'=D_m-\sum_{i=1}^NP_i,
\quad
D_m''=D_m'+P_s.
\end{equation}
This yields 
\begin{equation}\label{E:dimd}
\dim L'(D'_m)=l-rN,
\qquad
\dim L'(D''_m)=l-rN+r.
\end{equation}
The element in $L'(D'_m)$ have  orders $\ge (m+1)$ at all points in
$I$.
The elements in $L'(D''_m)$ have orders $\ge (m+1)$ at all points 
$P_i$, $i\ne s$ and orders $\ge m$ at $P_s$. 
{}From the dimension formula \refE{dimd} we conclude
that there exists $r$ elements which have exact order $m$ at $P_s$ and
orders
$\ge (m+1)$ at the other points in $I$. 
This says that there is for every basis element $X^u$ in the
Lie algebra $\g$ an element  $ X^u_{m,s}\in \gb$ which has exact
order $m$ at the point 
$P_s$ and order higher than $m$ at the other points in $I$
and can be written there as required in \refE{pal1}.
By linearity we get the  statement for all $X\in\g$.
\end{proof}
\begin{remark}
{\bf 1.} By modifying  the divisor 
$T$ in its degree we can 
even show that there exists elements  
such that the orders of $X_{m,s}$ at
the points $P_p$, $p\ne s$ are equal to $m+1$.
\newline
{\bf 2.}
We remark that for this proof no genericity arguments,
neither with respect to the points $A$ and $W$, nor with
respect to the parameter $\al_s$  were used.
Hence, the statement is true for all situations.
\newline
{\bf 3.}
In the very definition of $X_{m,s}$ the local coordinate
$z_s$ enters. In fact it only depends on the first order jet
of the coordinate, two different elements will just differ
by a rescaling.
\newline
{\bf 4.}
The elements $X_{m,s}$ are highly non-unique.
For introducing the almost-grading we will have to make them
essentially unique by trying to find a divisor $T$ as small
as possible but such that the statement is still true.
Further down, we will come back to this.
\end{remark} 
\begin{proposition}\label{P:filtbas1}
$ $
Let $X^u, u=1,\ldots,\dim\g$ be a basis of $\g$ and
\begin{equation}\label{E:filgen}
X^u_{m,s}, \quad u=1,\ldots,\dim\g, 
\quad s=1,\ldots,N,\quad
m\in\Z
\end{equation}
any fixed set of elements chosen according to
\refP{filtbas} then

\smallskip\noindent
(a) These elements are linearly independent.

\smallskip\noindent
(b) 
The set of classes
$[X^u_{m,s}]$, $u=1,\ldots,\dim\g$, $s=1,\ldots,N$ 
will constitute a basis of the quotient
$\gb_{(m)}/\gb_{(m+1)}$.  

\smallskip\noindent
(c)\quad 
$\dim\gb_{(m)}/\gb_{(m+1)}=N\cdot \dim\g$. 

\smallskip\noindent
(d) 
The classes of the elements $X^u_{m,s}$  will not depend on the elements
chosen. 
\end{proposition}
\begin{proof}
By the local expansion it follows like in the proof of \refP{basis}
that the elements \refE{filgen} are linearly independent, hence (a).
Furthermore, by ignoring higher orders, i.e. elements from
$\gb_{(m+1)}$ 
they stay linearly independent. Hence (b), and (c) follows. 
Part (d) is true by the very definition of the elements.
\end{proof}
Given $X\in\g$ we will denote 
for the moment  by $X_{m,s}$ any element fulfilling 
the conditions in \refP{filtbas}.

As the 
proof of \refP{almprod} stays also valid for these elements 
we have
\begin{proposition}
The algebra $\gb$ is a filtered algebra with respect to the
introduced filtration $(\gb_{(m)})$ i.e.
\begin{equation}
[\gb_{(m)},\gb_{(k)}]\subseteq \gb_{(m+k)}.
\end{equation}
Moreover,
\begin{equation}\label{E:alalg}
[X_{k,s},Y_{m,p}]={[X,Y]}_{k+m,s}\delta_s^p +L, \quad
L\in\gb_{(m+k+1)}.
\end{equation}
\end{proposition}

\bigskip
Our next goal is to introduce the homogeneous
subspaces $\gb_m$. A too naive method would be to take the
linear span of a fixed set of elements \refE{filgen}
for $\gb_m$.
The condition of almost-gradedness with respect to the
lower bound would be fulfilled by $m+k$, but not
necessarily for the upper bound.
To fix this we have to place more strict conditions on the pole orders
at $O$, and  we have to specify the divisor $(D_{m})_O$ in a coherent 
manner (with respect to $m$).
By our recipe the elements will become essentially unique
in the generic situation at least for nearly all $m$.
For non-generic $\al_s$ it might be necessary to modify the
prescription for individual component functions. But all these
modifications will change only the upper bound by a constant.

\begin{remark}
Before we advance we recall that for $\A$ and for the usual
Krichever-Novikov
current algebra $\g\otimes\A$ we have 
an almost-graded structure.
\newline
{\bf 1.}
As explained in \refS{algebras} we have the direct sum decomposition
\refE{glsplit}. Moreover, $\snb(n)\cong\A$. Hence the scalar part
is almost-graded and fulfills \refT{almgrad} and \refP{locex}.
If we show the statements for $\slnb(n)$ then it will follow for
$\glb(n)$.
Hence, it is enough to consider in the following the case of
$\g$ simple.
\newline
{\bf 2.}
Moreover, if the Tyurin data is empty 
(or all $\al_s=0$) then our Lax operator algebras reduce to the
Krichever-Novikov current algebras. For those we have the statements.
Hence, it is enough to consider Lax operator algebras with non-empty
Tyurin data. The reader might ask why we make such 
a different treatment.
In fact, for non-empty Tyurin data the proof will need less 
case distinctions.
\end{remark}

We will now give the general description for the  {\bf generic situation}
for $\g$ simple, 
and proof the claim about almost-gradedness in detail.
For the non-generic situation we will show where things have to be
modified.

Recall that 
for the divisor $D_m$ we had the decomposition
\refE{div_m}.
The terms  $(D_m)_I$ and $D_W$ stay as above.
For $(D_m)_O$ we require
\begin{equation}\label{E:conds}
(D_m)_O=\sum_{i=1}^M(a_im+b_{m,i})\,Q_i,
\end{equation}
with $a_i,b_{m,i}\in\Q$ such that $a_im+b_{m,i}\in\Z$,
$a_i>0$ and that there exists a $B$ such that 
$|b_{m,i}|<B, \forall m\in\Z, i=1,\ldots,M$.
Furthermore,
\begin{equation}\label{E:condd}
\begin{gathered}
\sum_{i=1}^Ma_i=N,\qquad 
\sum_{i=1}^Mb_{m,i}=N+g-1,\qquad
(D_{m+1})_O>(D_m)_O.
\end{gathered}
\end{equation}
For the degrees we calculate
\begin{equation}\label{E:degd}
\deg((D_m)_O)=m\cdot N+(N+g-1),
\quad 
\deg((D_{m+1})_O)=
\deg((D_m)_O)+N.
\end{equation}

\begin{example}
{\bf 1.} For  $M=1$ we have the unique solution
\begin{equation}
 (D_m)_O=(N\cdot m +(N+g-1)\,Q_M.
\end{equation}
\newline
{\bf 2.} For $N\ge M$  
the prescription
\begin{equation}\label{E:standnm}
(D_m)_O= (m+1)\sum_{j=1}^{M-1}Q_j+\big((N-M+1)(m+1)+g-1\big)Q_M
\end{equation}
will do. 
Apart from the  
$D_W$ the corresponding divisor $D_m$ was introduced 
in \cite{SLc} where the almost-grading in case of multi-point
Krichever-Novikov algebras and tensors has been considered for the
first time (see also \cite{Sad}). 
\newline
{\bf 3.} In  \cite{SLc} also prescriptions for the case $N<M$ were given.
We will not reproduce it here.
\end{example}
Hence in all cases we can find such divisors.

Now we set
\begin{equation}\label{E:gmdef}
  \gb_m\ :=\ \{L\in\gb\ | (L)\ge -D_m\}.
\end{equation}
\begin{proposition}\label{P:alm1}
$ $

\smallskip\noindent
(a)
\quad $\dim\gb_m=N\dim\g$.

\smallskip\noindent
(b) 
A basis of $\gb_m$ is given by elements $X_{m,s}^u$,
$u=1,\ldots, \dim\g$, $s=1,\ldots, N$ 
fulfilling the conditions
\begin{equation}\label{E:palu}
{X_{m,s}^u}_|(z_p)=X^uz_s^m\,\delta_p^s+O(z_p^{m+1}).
\end{equation}
\end{proposition}
\begin{proof}
We set $r:=\dim\g$. 
First we deal with the generic situation.
As explained above  
at the weak singular points we have exactly as much relations as we
get parameters by the poles. 
Hence for the calculation of $\dim L'(D)$ the contribution of the
degree of $D_W$  (which is $\epsilon \cdot K$) will be canceled
by the relations (which are $r\cdot\epsilon\cdot K$).
Here $\epsilon$ is equal to 1 or 2, depending on $\g$.
For the degree of $D_m$ we calculate
\begin{equation}
\deg D_m= g+(N-1)+\epsilon K\ge g.
\end{equation}
We stay in the region where equality 
for \refE{rrab} is true and calculate
\begin{equation}
\dim L'(D_m)=\dim L(D_m)-\epsilon rK=rN+r\epsilon K-\epsilon rK=rN.
\end{equation}
As by definition $\gb_m=L'(D_m)$ we get (a).
\newline
Next we consider $D_m'=D_m-\sum_{i=1}^NP_i$. For its
degree we calculate $\deg(D_m-\sum_{i=1}^NP_i)=g-1+\epsilon K$.
As $K\ge 1$ we are still in the domain where we have equality for
Riemann-Roch. Hence 
$\dim L'(D_m')=0$. Now for $D_m''=D_m'+P_s$ we calculate 
$\dim L'(D_m'')=r$.
This shows that for every basis element $X^u$ of  $\g$ 
there exists up to multiplication with 
a scalar a unique element $X_{m,s}^u\in\gb_m$ 
which has the local expansion
\begin{equation}
{X_{m,s}^u}_|(z_p)=X^u\delta_s^pz_p+O(z_p^{m+1}).
\end{equation} 
Hence, (b).
\newline
In the non-generic case we have to change the pole orders in 
the definition of the divisor 
part $(D_m)_O$ in a minimal way
by adding or subtracting finitely many points to reach
the situation such that we obtain exactly the dimension formula
and existence of the basis of the required type.
We have to take care that the number of changes maximally needed
will be bounded independent of $m$.
In fact this number is bounded by the number of points 
$Q$ from $O$ needed to add to 
the  divisor $D_m$  of the generic situation
(which is of degree $N+g-1+\epsilon K$) to reach
a divisor $D_m'$ with $\deg D_m'\ge 2g-1+H$, where $H$ is the number
of relations for the $\alpha_s$.
\end{proof}

\begin{proposition}\label{P:alm2}
\begin{equation}\label{E:decomp}
\gb=\bigoplus_{m\in\Z}\gb_m.
\end{equation}
\end{proposition}
\begin{proof}
The elements $X_{m,s}^u$ introduced as the basis elements in $\gb_m$
are elements of the type of \refP{filtbas} with respect to the
grading.
By \refP{filtbas1} they stay linearly independent even if 
we considered all $m$'s together, as their classes are linearly 
independent. Hence, the sum on the r.h.s. of \refE{decomp}
is a direct sum.
\newline
To avoid to take care of special adjustments to be done for
the non-generic situations we consider
$m\gg 0$ and  the divisor
\begin{equation}
E_m:=-(D_m)_I+D_W+(D_m)_O
=m\sum_{i=1}^NP_i+D_W+(D_m)_O,
\end{equation}
where $(D_m)_O$ is the divisor used for fixing the basis elements in
$\gb_m$, see \refE{conds}.
For its degree we have
\begin{equation}
\deg E_m=2mN+(N+g-1)+\epsilon K.
\end{equation}
For $m\gg 0$ we are in the region where \refE{rrab} is an equality.
Hence, after subtraction the relations we get 
\begin{equation}
\dim L'(E_m)=\dim\g\cdot\left((2m+1)N\right).
\end{equation}
The basis elements 
\begin{equation}\label{E:ela}
X_{k,s}^u, \qquad u=1,\ldots,\dim\g, \quad s=1,\ldots, N,
\quad -m\le k\le m
\end{equation}
are in $L'(E_m)$. This is shown by considering the orders at $I$ and
$O$. For $I$ it is obvious. For $O$ we have to use from
\refE{condd} the fact that $(D_{(k+1)})_O>(D_{k})_O$.
Hence, $-(D_m)_O$ is a lower bound for the $O$-part of the 
divisors for the element \refE{ela}. 
But these are $(2m+1)\cdot N\cdot \dim\g$ linearly independent elements. Hence,
\begin{equation}\label{E:elb}
L'(E_m)=
\bigoplus_{k=-m}^m\gb_k.
\end{equation}
An arbitrary element $L\in\gb$ has only finite pole orders at the
points in $I$ and $O$. Hence, there exists an $m$ such that
$L\in L'(E_m)$. This is again obvious for 
the points in $I$. For the points in $O$ we use that by the
conditions for $(D_m)_O$, see \refE{conds} for all $i=1,\ldots, M$
we have that $a_i>0$. Hence every pole order at $O$ will be superseded by a 
$(D_m)_O$ with $m$ suitably big. This shows the claim.
\end{proof}

\begin{proposition}\label{P:alm3}
There exist a constant $S$ independent of $n$ and $m$ such that
\begin{equation}\label{E:alm}
[\gb_m,\gb_k]\subseteq \bigoplus_{h=m+k}^{m+k+S}\gb_h.
\end{equation}
\end{proposition}
\begin{proof}
We will give the proof for the generic case (and $\g$ simple)
first and then point out the modification needed for
the general situation.
Let $L\in [\gb_m,\gb_k]$ 
then 
\begin{equation}\label{E:all}
 (L)\ge -(D_m+D_k)_I-D_W-(D_m+D_k)_O
\end{equation}
(observe that $D_W$ does not redouble here).
We consider the divisors $D_h$. Recall the formula 
\refE{conds}. As all $a_i>0$ there exists an $h_0$ such that 
$\forall h\ge h_0$ we have
\begin{equation}\label{E:ab1}
(D_h)_O\ge (D_m+D_k)_O 
\end{equation}
Hence, there exists also a smallest $h\in\Z$ such that 
\refE{ab1} is still true. We call this $h_{max}$. Again by
\refE{condd} $h_{max}\ge m+k$.
Now we consider the divisor
\begin{equation}
E_m=(D_m+D_k)_I+D_W+(D_{h_{max}})_O.
\end{equation}
{}From \refE{degd} we calculate
\begin{equation}
\deg((D_{hmax})_O)=
\deg((D_m+D_k)_O)+(h_{max}-(m+k))N.
\end{equation}
Hence,
\begin{equation}
\deg(E_m)=-(m+k)N+\epsilon K+h_{max}\cdot N+(N+g-1).
\end{equation}
As $\deg(E_m)\ge g$ and under the assumption of genericity we
stay in the region where 
\begin{equation}
\dim L'(E_m)=\deg\g\cdot(h_{max}-(m+k)+1)N.
\end{equation} 
As in the proof of \refP{alm2} we get that the elements
\refE{ela} for $m+k\le h\le h_{max}$ lie in $L'(E_m)$. 
They are linearly independent, hence 
\begin{equation}\label{E:mhmax}
L'(E_m)=\bigoplus_{h=n+m}^{h_{max}}\gb_h.
\end{equation}
By \refE{all} the $L$, we started with, lies also in $L'(E_m)$
and consequently also on the right hand side of \refE{mhmax}.
\newline
To show almost-grading we have to show that 
there exists an $S$ (independent of $m$ and $k$ such that
$h_{max}=m+k+S$.
The relation \refE{ab1} can be rewritten as
\begin{equation}\label{E:ab2}
a_ih+b_{h,i}\ge a_i(m+k)+b_{m,i}+b_{k,i},\quad
\forall i=1,\ldots, M.
\end{equation}
This rewrites to
\begin{equation}
h\ge (m+k)+\frac{b_{m,i}+b_{k,i}- b_{h,i}}{a_i},
\quad \forall i=1,\ldots, M.
\end{equation}
The minimal $h$ for which this is true is
\begin{equation}\label{E:3t}
h_{max}= (m+k)+
\min_{i=1,\ldots, M}\lceil\frac{b_{m,i}+b_{k,i}- b_{h,i}}{a_i}\rceil,
\end{equation}
where for any real number $x$ the $\lceil x\rceil$ denotes the
smallest integer $\ge x$.
As our $|b_{m,i}|$ are bounded uniformly by $B$ 
the 3.term in \refE{3t} will be uniformly bounded by a constant $S$
too.
Hence, we get almost-grading. In the case of non-generic points and
$\al_s$'s the divisors at $O$ have to be modified by finitely many
modifications. Hence the constant $S$ has to be adapted by adding
a finite constant to it.
But still everything remains almost-graded.
\end{proof}
{}From the proof we can even calculate $h_{max}$ if needed.
As an example we give
 \begin{corollary}\label{C:estim}
In the generic simple Lie algebra case 
for $N\ge M$ with the standard prescription 
\refE{standnm} 
we have $\quad h_{max}=n+m+S\quad$ with 
\begin{equation}
S=\begin{cases}
0,&g=0,\ N=M=1,
\\
1,&g=0,\ M>1,
\\
1,&g=1
\\
1+\lceil\frac {g-1}{N-M+1}\rceil,&g\ge 2.
\end{cases}
\end{equation}
\end{corollary}
\begin{proof}
For the standard prescription we have
\begin{equation}
\begin{gathered}
a_i=1,\ i=1,\ldots, M-1,\quad
a_M=N-M+1,
\\
b_i=b_{m,i}=1,\  i=1,\ldots ,M-1,\quad
b_M=b_{m,M}=N-M+g.
\end{gathered}
\end{equation}
Hence,
\begin{equation}
S=\max_{i=1,\ldots,M}\lceil\frac {b_i}{a_i}\rceil.
\end{equation}
which yields the result.
\end{proof}

Now we are ready to collect the results of  Propositions
\ref{P:alm1}, \ref{P:alm2} and \ref{P:alm3}. The statements 
are exactly the statements both of \refT{almgrad} and
\refP{locex}. All statements of \refS{almstruct} are now
shown to be true.
In particular, now we know that both filtrations 
\refE{filt1} and \refE{filt} coincide. Hence, 
\refE{filt} is also canonically defined by the splitting of
$A$ into $I$ and $O$.

\bigskip

A Lie algebra $\V$ is called {\it perfect} if $\V=[\V,\V]$. 
Simple Lie algebras are of course perfect.
The usual
Krichever-Novikov current algebras $\gb$ for $\g$ simple are
perfect too \cite[Prop. 3.2]{Saff}.
Lax
operator algebras are not 
necessarily perfect
(at least we do not have  a proof of it).
Lemma \ref{L:weak} below might be
considered as a weak analog of that property.

\begin{lemma}\label{L:weak}
Let $\g$ be simple and $y\in\gb$ then for every $m\in\Z$ there
exists finitely many elements $y^{(s,1)}, y^{(s,2)}\in\gb$,
$i=1,\ldots, l=l(m)$ such that
\begin{equation}
y-\sum_{s=1}^l\;[y^{(s,1)}, y^{(s,2)}]\quad\in\quad \gb_m.
\end{equation}
\end{lemma}
\begin{proof}
Let $y$ be an element of $\gb$. Hence there exists a $k$ such that
$y\in\gb_{(k)}$, but $y\ne \gb_{(k+1)}$. In particular there exists
for every point $P_i$ elements $X_{k,i}^i$ such that
\begin{equation}
y-\sum_{i=1}^NX_{k,i}^i\in\gb_{(k+1)},
\end{equation}
where $X_{k,i}^i=(X^i)_{k,i}$ is the element corresponding to
$X^i\in\g$.
As $\g$ is perfect we have 
$X^i=[Y^i,Z^i]$ with  elements $Y^i,Z^i\in\g$.
We calculate
\begin{equation}
X_{k,i}^i=[Y_{0,i}^i,Z_{k,i}^i]+y^{i},\quad
 y^{(i)}\in\gb_{(k+1)}.
\end{equation}
In total
\begin{equation}
y^{(k)}=y-
\sum_{i=1}^N[Y_{0,i}^i,Z_{k,i}^i]\in\gb_{(k+1)}.
\end{equation}
Using the same again for $y^{(k)}$ etc., we can approximate 
$y$ to every finite order by sums of commutators. 
\end{proof}
\section{Module structure }\label{S:mod}
$ $

\subsection{Lax operator algebras as modules 
over $\A$}\label{S:Amod}
$ $

The space $\gb$ is an $\A$-module with respect to the point-wise
multiplication. Obviously, the relations \refE{glexp},
\refE{gldef}, \refE{sodef}, \refE{spdef}, are not disturbed. 
\begin{proposition}
$ $

\smallskip\noindent
(a) The Lax operator algebra $\gb$ is an almost-graded module over
$\A$, i.e. 
there exists a constant $S_4$ (not depending on $k$ and $m$)
such that
\begin{equation}\label{E:amod}
\A_k\cdot\gb_m\subseteq\bigoplus_{h=k+m}^{k+m+S_4} \gb_h.
\end{equation}

\smallskip\noindent
(b) For $X\in\g$
\begin{equation}
A_{m,s}\cdot X_{n,p}=X_{m+n,s}\,\delta_p^s+L,\qquad L\in \gb_{(m+n+1)}.
\end{equation}
\end{proposition}
\begin{proof}
We consider the  orders of the elements in $I$ and $O$.
As in the proof of \refP{alm3}  the existence of
a constant $S_4$ follows so that \refE{amod} is true. Hence (a).
\newline
We study the lowest order term  of 
$A_{m,s}\cdot X_{n,r}$ at the points $P_i\in I$. Using \refE{pal1},
\refE{Aset},\refE{Anorm} we see that if $s\ne r$ then
$A_{m,s}\cdot X_{n,r}\in \gb_{(m+n+1)}$ as all orders are $\ge n+m+1$.
The same is true for $s=r$ for the element
$A_{m,s}\cdot X_{n,s}-X_{m+n,s}$. Hence the claim.
\end{proof}
Warning: in general we do not 
have $A_{m,s}\cdot X_{0,s}=X_{m,s}$ as the orders at
$O$ do not coincide. Also, $A_{m,s} \cdot X$ does not necessarily 
belong to $\gb$. 

\subsection{Lax operator algebras as modules 
over $\L$}\label{S:Lmod}
$ $

Next we introduce an  action of $\L$ on $\laxg$. 
This is done with the help of a certain connection 
$\nabla^{(\omega)}$ following the lines of
\cite{Klax}, \cite{Kiso}, \cite{KSlax} 
with the modification made in \cite{SSlax}.
The
connection form $\w$ is   a $\g$-valued meromorphic 1-form,
holomorphic outside $I$, $O$ and $W$, and has a certain
prescribed behavior at the points in $W$. For $\gamma_s\in W$ with
$\a_s= 0$ the requirement is that $\w$ is also regular there. For
the points $\gamma_s$ with   $\a_s\ne 0$ it is  required that it has
an  expansion of the form
\begin{equation}\label{E:connl}
\w(z_s)=\left(\frac {\w_{s,-1}}{z_s}+\w_{s,0}+\w_{s,1}+
\sum_{k>1}\w_{s,k}z_s^k\right)dz_s.
\end{equation}
For 
$\gl(n)$: there exist $\tb_s\in\C^n$ and $\tk_s\in \C$
such that
\begin{equation}\label{E:gldefc}
\w_{s,-1}=\a_s \tb_s^{t},\quad \w_{s,0}\,\a_s=\tk_s\a_s, \quad
\tr(\w_{s,-1})=\tb_s^t \a_s=1.
\end{equation}
For $\so(n)$: there exist $\tb_s\in\C^n$ and
$\tk_s\in \C$ such that
\begin{equation}\label{E:sodefc}
\w_{s,-1}=\a_s\tb_s^t-\tb_s\a_s^t, \quad
\w_{s,0}\,\a_s=\tilde\ka_s\a_s, \quad \tb_s^t\a_s=1.
\end{equation}
For $\spn(2n)$: there exist $\tb_s\in\C^{2n}$,
$\tilde\ka_s\in\C$ such that
\begin{equation}\label{E:spdefc}
\w_{s,-1}=(\a_s\tb_s^t+\tb_s\a_s^t)\sigma, \quad
\w_{s,0}\,\a_s=\tilde\kappa_s\a_s, \quad
\a^t_s\sigma\w_{s,1}\a_s=0,\quad \tb_s^t\sigma\a_s=1.
\end{equation}
The existence of nontrivial connection forms fulfilling the listed
conditions is proved 
by  Riemann-Roch type argument as \refP{filtbas}.
We might even require, and
actually always will do so, that the connection form is holomorphic
at $I$. Note also that if all $\a_s=0$ we could take $\w=0$.

The connection form $\omega$ induces the following
connection $\nabla^{(\omega)}$  on
$\gb$ 
\begin{equation}\label{E:conng}
\nabla^{(\w)}=d+[\w,.].
\end{equation}
Let $e\in\La$
be a vector field. In a local coordinate $z$ the connection form
and the vector field are represented as $\omega=\tilde\omega dz$
and
 $e=\tilde e\frac{d}{dz}$
with a local function $\tilde e$ and a local matrix valued
function  $\tilde\omega$. The covariant derivative in direction of
$e$ is given by
\begin{equation}\label{E:covder}
\nabla_e^{(\w)}=dz(e)\frac {d}{dz}+[\w(e),.\,]= e\ldot
+[\,\tilde\omega\tilde e\, ,.\,] =\tilde e\cdot \big(\frac
{d}{dz}+[\,\tilde\omega\, ,.\,]\big).
\end{equation}
Here the first term $(e.)$ corresponds to taking the usual derivative of
functions in each matrix element separately, whereas 
$\tilde e\cdot$ means multiplication with the local function
$\tilde e$.

Using the last description we obtain  for $L\in\gb,\
g\in \A,\  e,f\in\L$
\begin{equation}\label{E:conn1r}
\nabla_e^{(\w)}(g\cdot L)=(e\ldot g)\cdot L + g\cdot
\nabla_e^{(\w)}L, \qquad \nabla_{g\cdot e}^{(\w)}L=g\cdot
\nabla_e^{(\w)}L,
\end{equation}
and
\begin{equation}\label{E:conn2r}
\nabla_{[e,f]}^{(\w)}=[\nabla_{e}^{(\w)},\nabla_{f}^{(\w)}].
\end{equation}
The proofs of the following statements  are completely the
same as the proofs for the two-point case 
presented in \cite{SSlax}. Hence, they are
here omitted.
\begin{proposition}\label{P:deri}
$ $

\smallskip\noindent
(a) $\nabla_e^{(\w)}$  acts as a derivation on the Lie algebra
$\laxg$, i.e.
\begin{equation}\label{E:deri}
\nabla_e^{(\w)}[L,L']= [\nabla_e^{(\w)}L,L']+
[L,\nabla_e^{(\w)}L'].
\end{equation}

\smallskip\noindent
(b) The covariant derivative makes $\laxg$  to  a Lie module over
$\L$.

\smallskip\noindent
(c) 
The decomposition $\glb(n)=\snb(n)\oplus\slnb(n)$ is a
decomposition into $\L$-modules, i.e.
\begin{equation}
\nabla_e^{(\w)}: \snb(n)\to \snb(n),\qquad \nabla_e^{(\w)}: \slnb(n)\to
\slnb(n).
\end{equation}
Moreover,  the $\L$-module
$\snb(n)$ is equivalent to the  $\L$-module $\A$.
\end{proposition}

\begin{proposition}\label{P:almgrad}
$ $

\smallskip\noindent
(a)  $\gb$  is an almost-graded $\L$-module.

\smallskip\noindent
(b) For the corresponding $\L$-action we have
\begin{equation}\label{E:allstr}
\nabla_{e_{k,s}}^{(\w)}X_{m,r}=m\cdot X_{k+m,s}\,
\delta_s^r+L, \quad  L\in \gb_{(k+m+1)}.
\end{equation}
\end{proposition}
\begin{proof}
(a) 
By \refP{deri} $\gb$ is an $\L$-module. It remains to show that
there is an upper bound for the order of the elements of the
type $n+m+S_5$, with $S_5$ independent of $n$ and $m$ (but may depend
on $\w$).
We write \refE{covder}  for homogeneous elements
and obtain 
\begin{equation}\label{E:covex}
\nabla_{e_{k,s}}^{(\w)}X_{m,r}=e_{k,s}\ldot X_{m,r}+[\,\tilde\omega\tilde
  e_{k,s}\, ,
X_{m,r}].
\end{equation}
The form $\w$ has fixed orders at $I$  and at $O$, the action of
$\L$ on $\A$ is almost-graded, and the bracket corresponds to the
commutator in the almost-graded $\gb$. 
By considering the  corresponding bounds for the order of poles at
$I$ and $O$ we get such an universal bound.
\newline
(b) Locally at $P_i$, $i=1,\ldots, N$ we have
\begin{equation}
{X_{m,r}}_|(z_i)=Xz_i^m\delta_i^r+O(z_i^{m+1}),\qquad {e_k}_|(z_i)
=z_i^{k+1}\delta_i^k\frac
{d}{dz}+O(z_i^{k+2}).
\end{equation}
This implies
\begin{equation}
e_{k,s}\ldot X_{m,r}(z_i)=m X z_i^{k+m}\delta_i^r\delta_i^s+O(z_i^{k+m+1}), \quad
\tilde\omega\tilde e_k(z_i)=Bz_i^{k+1}+O(z_i^{k+2}),
\end{equation}
with $B\in\gl(n)$. Hence
\begin{equation}
[\,\tilde\omega\tilde e_k\, ,X_m]=O(z_i^{k+m+1}), \forall i,
\end{equation}
and the second term will only contribute to higher order.
It remains the first term in \refE{covex}.
If $r\ne s$ then 
$e_{k,s}\ldot X_{m,r}(z_i)\in O(z_i^{k+m+1})$ for all $i$.
If $r=s$ then 
$(e_{k,s}\ldot X_{m,s}-mX_{m+k,s})(z_i))\in O(z_i^{k+m+1})$.
Hence,  \refE{allstr} follows.
\end{proof}

\subsection{Module structure over $\D$ 
and the algebra $\D_{\g}$
} \label{S:d1alg} $ $

The Lie algebra  $\D$ of meromorphic differential operators on
$\Sigma$ of degree $\le 1$ holomorphic outside of $I\cup O$ is
defined as the semi-direct sum of $\A$ and $\L$ with the
commutator between them given by the action of $\L$ on $A$. 
It is the vector space direct sum 
$\D=\A\oplus\L$ with Lie bracket
\begin{equation}
[(g,e),(h,f)]:=(e\ldot h-f\ldot g,[e,f]).
\end{equation}
In particular
\begin{equation}\label{E:relsd}
[e,h]=e\ldot h.
\end{equation}
It is an almost-graded Lie algebra \cite{Scocyc}.
\begin{proposition}
The Lax operator algebras $\laxg$ are almost-graded Lie modules
over $\D$ via
\begin{equation}
e\ldot L:=\nabla_e^{(\w)}L,\qquad h\ldot L:=h\cdot L.
\end{equation}
\end{proposition}
\begin{proof}
As $\laxg$ is an almost-graded $\A$- and $\L$-module it is enough
to show that the relation \refE{relsd} is satisfied. For $e\in \L,
h\in\A, L\in\laxg$ using \refE{covder} we get
\begin{multline*}
e\ldot (h\ldot L)-h\ldot(e\ldot L)= \nabla_e^{(\w)}(h L)-h
\nabla_e^{(\w)}(L) =
\\
\tilde e\left(\frac {d(hL)}{dz}+[\tilde \w, hL]\right) -h\tilde
e\left(\frac {dL}{dz}+[\tilde \w, L]\right) = \left(\tilde e\frac
{dh}{dz}\right)L= (e\ldot h)L=[e,h]\ldot L.
\end{multline*}
\end{proof}

The Lax operator
algebra $\gb$ is a module over the Lie algebra $\L$ which acts on
$\gb$ by derivations (according to \refP{deri}). \refP{deri} says
that this action of $\L$ on $\laxg$ is an action by derivations.
Hence as above we can consider the semi-direct sum
$\D_{\g}=\laxg\oplus\L$ with Lie product given by
\begin{equation}\label{E:dac}
[e,L]:=e\ldot L=\nabla_e^{(\w)}L,
\end{equation}
for the mixed pairs. See \cite{Saff} for the corresponding
construction for the classical Krichever-Novikov algebras of
affine type. 
\section{Cocycles}
In this section we will study 2-cocycles  for 
the Lie
algebra $\laxg$ with values in $\C$.
It is well-known that the  corresponding cohomology
space $\H^2(\laxg,\C)$ classifies equivalence classes of
(one-dimensional) central extensions of $\laxg$.

For the convenience of the reader we recall that
 a 2-cocycle for $\laxg$ is a bilinear form
$\ga:\laxg\times\laxg\to\C$ which is (1) antisymmetric and (2)
fulfills the condition
\begin{equation}\label{E:cohcoc}
\ga([L,L'],L'')+ \ga([L',L''],L)+ \ga([L'',L],L')=0,
\qquad L,L',L''\in\gb\;.
\end{equation}
A 2-cocycle $\ga$ is a coboundary if there exists a linear form
$\phi$ on $\laxg$ such that
\begin{equation}
\ga(L,L')=\phi([L,L']),\qquad  L,L'\in\laxg.
\end{equation}

Given a 2-cocycle $\ga$ for $\laxg$, the associated
central extension $\gh_{\ga}$ is given as vector space direct sum
$\gh_\ga=\gb\oplus\C\cdot t$ with Lie product 
\begin{equation}\label{E:centextf}
[\widehat{L},\widehat{L'}]=\widehat{[L,L']}+\ga(L,L')\cdot t,
\quad [\widehat{L},t]=0,\qquad L,L'\in\laxg.
\end{equation}
Here we used $\widehat{L}:=(L,0)$ and $t:=(0,1)$. Vice versa,
every central extension
\begin{equation}
\begin{CD}
0@>>>\C@>i_2>>\gh@>p_1>>\gb@>>>0,
\end{CD}
\end{equation}
defines  a 2-cocycle $\ga:\gb\to\C$ by choosing a section
$s:\gb\to \gh$.

Two central extensions $\gh_{\ga}$ and $\gh_{\ga'}$ are equivalent
if and only if the defining cocycles $\ga$ and $\ga'$ are cohomologous.

\label{S:cocycles}
\subsection{Geometric cocycles}
$ $

Next we introduce geometric 2-cocycles. 
Let $\w$ be a connection form as introduced in \refS{Lmod}
for defining the connection \refE{conng}. Furthermore, let $C$ be
a (not necessarily connected ) 
differentiable cycle on $\Sigma$ not meeting the sets 
$A=I\cup O$ and $W$. 

As in  the two point situation considered in 
\cite{SSlax}
 we define the following bilinear forms on $\laxg$:
\begin{equation}\label{E:g1}
\ga_{1,\w,C}(L,L')= \cinc{C} \tr(L\cdot \nabla^{(\w)}L'), \qquad
L,L'\in\laxg,
\end{equation}
and
\begin{equation}\label{E:g2}
\ga_{2,\w,C}(L,L')= \cinc{C} \tr(L)\cdot \tr(\nabla^{(\w)}L'),
\qquad L,L'\in\laxg.
\end{equation}
The following propositions and their proofs 
remain the same as in \cite{SSlax}
(of course now to be interpreted in this more general context),
and we will not repeat them.
\begin{proposition}\label{P:c12cc}
The bilinear forms $\ga_{1,\w,C}$ and $\ga_{2,\w,C}$ are cocycles.
\end{proposition}
\begin{proposition}\label{P:cind}
$ $

\noindent (a) The cocycle  $\ga_{2,\w,C}$ does not depend on the
choice of the connection form $\w$.

\noindent (b) The cohomology class  $[\ga_{1,\w,C}]$ does not
depend on the choice of the connection form $\w$. More precisely
\begin{equation}\label{E:cobw}
\gamma_{1,\w,C}(L,L')-\gamma_{1,\w',C}(L,L')= \cinc{C}\tr\big((\w-
\w')[L,L']\big).
\end{equation}
\end{proposition}

As $\ga_{2,\w,C}$ does not depend on $\w$ we will drop $\w$  in
the notation. Note that $\gamma_{2,C}$ vanishes on $\laxg$ for
$\g=\sln(n),\so(n),\spn(2n)$. But it does not vanish on
$\laxs(n)$, hence not on $\laxgl(n)$.

\subsection{$\L$-invariant cocycles}
$ $

As explained in \refS{Lmod} 
after fixing a connection form $\w'$ the vector field
algebra $\L$ operates on $\laxg$ via the covariant derivative
$e\mapsto \nabla^{(\w')}_e$. 

\begin{definition}\label{E:linv}
A cocycle $\ga$ for $\laxg$ is called {\it $\L$-invariant} (with respect
to $\w'$)  if
\begin{equation}
\gamma(\nabla^{(\w')}_{e}L,L')+
\gamma(L,\nabla^{(\w')}_{e}L')=0,\qquad \forall e\in\L,\quad
\forall L,L'\in \laxg.
\end{equation}
\end{definition}

\begin{proposition}\label{P:linv}

(a) The cocycle $\gamma_{2,C}$ is $\L$-invariant.

(b) If $\w=\w'$ then the cocycle $\gamma_{1,\w,C}$ is
$\L$-invariant.
\end{proposition}
The proof is the same as presented in \cite{SSlax} for the 
two point case.

\medskip
We call a cohomology class {\it $\L$-invariant} if it has a
representing cocycle which is  $\L$-invariant. The reader should
be warned that this does not mean that all representing cocycles
are  $\L$-invariant. On the contrary, see \refC{uni}.
Clearly, the  $\L$-invariant classes constitute a
subspace of
 $\H^2(\laxg,\C)$ which we denote by
 $\H^2_{\L}(\laxg,\C)$.
\subsection{Some remarks on the cocycles
on $\D_\g$}\label{S:remarks} $ $

In the following let $\w=\w'$. The property of $\L$-invariance of
a cocycle has  a deeper meaning.
In \refS{d1alg} we introduced the algebra
$\D_\g$. The Lax operator algebra $\gb$ is a subalgebra of $\D_\g$.
Given a 2-cocycle $\ga$ for $\laxg$ we might extend  it to
$\D_\g$ as a bilinear form by setting ($L,L'\in\laxg$, $e,f\in\L$)
\begin{equation}\label{E:323}
\tilde\ga(L,L')=\ga(L,L'),\quad
\tilde\ga(e,L)=\tilde\ga(L,e)=0,\quad \tilde\ga(e,f)=0.
\end{equation}
\begin{proposition}
 The extended  bilinear form $\tilde\ga$ is a cocycle for  $\D_\g$ if and only
if $\ga$ is $\L$-invariant.
\end{proposition}
\begin{proof}
The conditions defining a cocycle are obviously fulfilled for the
triples of elements consisting either of currents or of vector
fields. The only condition which does not follow automatically
from \refE{323} for $\tilde\ga$ is 
\begin{equation}\label{E:cod}
\tilde\ga([L,L'],e)+ \tilde\ga([L',e],L)+ \tilde\ga([e,L],L')=0.
\end{equation}
Using \refE{dac} we get that \refE{cod} is true if an only if
\begin{equation}
\ga(\nabla_e^{(\w)}L,L')+ \ga(L,\nabla_e^{(\w)}L')=0,
\end{equation}
which is $\La$-invariance.
\end{proof}

\subsection{Bounded and local cocycles}\label{S:local}
$ $

\begin{definition}
Given an almost-graded Lie algebra $\V=\bigoplus_{m\in\Z}\V_m$. A 
cocycle $\ga$ is called {\it bounded (from above)} if 
there exists a constant $R_1\in\Z$ such that
\begin{equation}\label{E:bounded} 
\ga(\V_n,\V_m)\ne 0 \implies 
n+m\le R_1.
\end{equation}
\newline
Similarly bounded from below is defined.
\newline
A cocycle is called {\it local} if and only if it is bounded from
above and below. Equivalently,
there exist $R_1,R_2\in\Z$ such 
\begin{equation}\label{E:cloc}
\ga(\V_n,\V_m)\ne 0 \implies R_2\le n+m\le R_1. 
\end{equation}
\end{definition}
The
almost-grading of $\V$ can be extended  from $\V$ to the
corresponding central extension ${\widehat\V}_\ga$ \refE{centextf}
by assigning to the central element $t$ a certain degree (e.g. the degree
0) if and only if the 
defining cocycle for the central extension is local.

We call a cohomology class {\it bounded (resp. local)} if it contains a 
bounded (resp. local) representing 
cocycle. Again, not every representing cocycle of a bounded
(resp. local) 
class is
bounded (resp. local). The  set of bounded cohomology classes is a
subspace of  $\H^2(\laxg,\C)$ which we denote by
 $\H^2_{b}(\laxg,\C)$. 
It contains the subspace of local cohomology classes denoted by
$\H^2_{loc}(\laxg,\C)$.
This space classifies the almost-graded central extensions of
$\laxg$ up to equivalence. 
Both spaces admit subspaces consisting of 
those  cohomology classes admitting a representing cocycle
which is both 
bounded (resp. local) and $\L$-invariant.
The subspaces are denoted by
$\H^2_{b,\L}(\laxg,\C)$, resp.
$\H^2_{loc,\L}(\laxg,\C)$.

If we consider our geometric cocycles $\ga_{2,C}$ and
$\ga_{1,\w,C}$ obtained by integrating over an arbitrary
cycle then they will neither be bounded, nor local, nor will they
define a bounded or local cohomology class.

Next we will consider special integration paths.
Let $C_i$ be  positively oriented (deformed) circles around
the points $P_i$ in $I$, $i=1,\ldots, N$ 
 and $C_j^*$ positively oriented ones around
the points $Q_j$ in $O$, $j=1,\ldots, M$.
The cocycle values of $\ga$ if integrated over such cycles 
can be calculated via residues, e.g.
\begin{equation}\label{E:ga1}
\ga_{1,\w,C_i}(L,L')=\res_{P_i}(\tr(L\cdot \nabla^{(\w)} L')),
\quad i=1,\ldots, N\; .
\end{equation}

\begin{proposition}\label{P:locg1}
(1) The 1-form $\tr(L\cdot \nabla^{(\w)} L')$ has no poles outside of 
$A=I\cup O$.

(2) The 1-form $\tr(L)\cdot \tr(dL')$ has no poles outside of
$A=I\cup O$. 
\end{proposition}
\begin{proof} For (1) see \cite{KSlax}. For (2) see
\cite{SSlax}.
\end{proof}

A cycle $C_S$ is called a separating cycle 
if it is smooth, positively oriented of multiplicity one, 
it separates
the points in $I$ from the points in $O$, and it does not meet $A$ or $W$. 
It might have multiple components. 
For our cocycles \refE{g1}, \refE{g2}  
we integrate  
the forms of \refP{locg1} over closed
curves $C$. By this proposition the integrals will 
yield the same results if 
$[C]=[C']$ in  $\Ho(\Sigma\setminus A,\Z)$.
Note that the weak singular points will not show up in
this context.
In this sense
we can write 
for every separating cycle
\begin{equation}\label{E:cs}
[C_S]=\sum_{i=1}^K[C_i]=-\sum_{j=1}^M [C^*_j].
\end{equation}
The minus sign appears due to the opposite orientation.
In particular the cocycle values obtained by integrating over
a $C_S$ can be obtained by calculating residues either over the 
points in $I$ or the points in $O$.

\begin{theorem}
Let $\w$ coincides with the
connection form $\w'$ associated to the $\L$-action then 

\noindent
(a) For $i=1,\ldots, N$ the cocycles 
$\ga_{1,\w,C_i}$ and   $\ga_{2,C_i}$  with $C_i$ a circle around $P_i$
 will be bounded from above  and $\L$-invariant.

\noindent
(b) For $j=1,\ldots, M$ the cocycles 
$\ga_{1,\w,C_j^*}$ and  $\ga_{2,C^*_j}$  with $C^*_j$ a circle around $Q_j$
 will be bounded from below   and $\L$-invariant.

\noindent
(c) The cocycles 
$\ga_{1,\w,C_S}$ and  $\ga_{2,C_S}$  with $C_S$ a separating
cycle will be local and $\L$-invariant.

\noindent
(d) In case (a) and (c) the upper bound will be zero.
\end{theorem}

\begin{proof}
The statement about $\La$-invariance follows from \refP{linv}.
In fact only for this $\w=\w'$ is needed.
\newline
As explained above the cocycle calculation if integrated over $C_i$ 
(or over $C^*_j$) reduces to the calculation of residues.
Let $L\in\gb_n$, $L'\in\gb_m$ then 
$\ord_{P_i}(L)\ge n$ and $\ord_{P_i}(L')\ge m$. As $\w$ is holomorphic 
at $P_i$ we obtain
\begin{equation*}
\ord_{P_i}(dL')\ge m-1,\quad\ord_{P_i}(\nabla^{(\w)}L')\ge m-1\;.
\end{equation*}
Hence, if $n+m>0$ neither one of the 1-forms appearing in the
cocycle definition has poles at $I$ and consequently no residues.
This shows (a).
\newline
For (b) we have to consider the orders at the points in $O$ of the 
basis elements of $\gb_m$. By the prescriptions
\refE{conds} and \refE{condd} and taking into account possible 
poles of $\w$ at $O$ we find an $R_2$ such
 that if $n+m\le R_2$ the integrands will not
have poles anymore. This shows (b).
\newline
Using  \refE{cs} we can obtain the values of the cocycles integrated
over $C_S$ either by adding up the values obtained by integration
either over $I$ or over $O$. Hence boundedness from below and from
above.
Hence, locality.
\newline
That zero is an upper bound followed already during the proof.
\end{proof}

\section{Classification Results}
\label{S:class}
$ $

Recall that we are in the multi-point situation $A=I\cup O$ 
with $\#I=N$. The $C_i$, $C^*_j$, and $C_S$ are the special cycles
introduced in \refS{local}.
If we will use the word {\it bounded} for a cocycle we always mean
bounded from above if nothing else is said.
\begin{proposition}\label{P:linind}
The cocycles $\ga_{1,\w,C_i}$, $i=1,\ldots,N$ 
(and $\ga_{2,C_i}$, $i=1,\ldots, N$ for $\glb(n)$) are
linearly independent.
\end{proposition} 
\begin{proof}
Assume that there is a linear relation
\begin{equation}
0=\sum_{i=1}^N\al_i\ga_{1,\w,C_i}
+
\sum_{i=1}^N\beta_i\ga_{2,C_i},\quad \alpha_i,\beta_i\in\C.
\end{equation}
The  last sum will not appear in the simple algebra
case. Recall that for a  pair $L,L'\in\gb$ 
the above cocycles can be calculated by taking residues
\begin{equation}\label{E:sumlin}
0=\sum_{i=1}^N\al_i\res_{P_i}(\tr (L\cdot \nw L'))
+
\sum_{i=1}^N\beta_i\res_{P_i}(\tr(L)\cdot \tr(\nw L')).
\end{equation}
In the first sum the Cartan-Killing form is
present which is non-degenerated. Hence there exist $X,Y
\in\g$  such that $\tr(XY)\ne0$ and $\tr(X)=\tr(Y)=0$.
For $k=1,\ldots, N$, 
using the almost-graded structure and following 
\refP{locex} we take $L=X_{1,k}$ and $L'=Y_{-1,k}$.
In the neighbourhood of the point $P_l$, $l=1,\ldots, N$ we
have
\begin{equation}
\begin{gathered}
L(z_l)=Xz_l\delta_l^k+O(z_l^2),
\quad
L'(z_l)=Yz_l^{-1}\delta_l^k+O(z_l^0),
\\
\nw L'(z_l)=-Yz_l^{-2}\delta_l^k+O(z_l^{-1}),
\end{gathered}
\end{equation}
as $\nw L'=dL'+[\omega,L']$.
Hence,
\begin{equation}
\res_{P_l}(\tr(L\cdot\nw L'))= -\tr(XY)\delta_l^k.
\end{equation}
As $\tr(X)=0$ the second sum will vanish anyway and
we conclude $\alpha_k=0$, for all $k=1,\ldots, N$.
For the second sum we take $X=Y$ a nonvanishing scalar
matrix and chose $L=X_{1,k}$ and $L'=X_{-1,k}$. We  obtain 
 $\beta_k=0$ for all $k=1,\ldots, N$.
\end{proof}
\begin{proposition}
($\gb=\glb(n)$)
Let \
$\ga=\sum_{i=1}^N\beta_i\ga_{2,C_i}$  
be a nontrivial linear combination, then it is not a coboundary.
\end{proposition}
\begin{proof}
Recall from \refE{glsplit} that $\snb(n)$  is an abelian
subalgebra of $\glb(n)$. Hence, every coboundary 
restricted to it will be identically zero.
If we take again as in the previous proof elements
$X_{1,k}$ and $X_{-1,k}$ from the scalar subalgebra we 
obtain as above $\beta_k=0$.
\end{proof}

\begin{proposition}\label{P:nonbound}
Let $\ga=\sum_{i=1}^N\al_i\ga_{1,\w,C_i}$ be
a non-trivial linear combination then 
it is not a coboundary.
\end{proposition}
\begin{proof}
Assume that $\ga$ is a coboundary. This means that there exists a
linear form $\phi:\gb\to\C$ such that $\forall L,L'\in\gb$
\begin{equation}\label{E:contra}
\ga(L,L')=\sum_{i=1}^N\alpha_i\res_{P_i}\tr(L\cdot\nw L')= \phi([L,L']).
\end{equation}
Assume that $\ga\ne 0$, hence one of the coefficients $\alpha_k$
will be non-zero. 
Take $H\in\fh$ with $\kappa(H,H)\ne 0$, where $\fh$ is the Cartan
subalgebra of the simple part of $\g$ and $\kappa$  its
Cartan-Killing form. Let $H_{0,k}\in\gb$  be the element defined by
\refE{locex}. In particular, we have $H_{0,k}=H+O(z_k)$.
We set%
\footnote{Notice  that  $H_{(n,k)}$ and $H_{n,k}$, in general, are
different but coincide up to higher order.} $H_{(n,k)}:=H_{0,k}\cdot
A_{n,k}\in\gb$ and hence $H_{(n,k)}=H\cdot A_{n,k}+O(z_k^{n+1})$
in the neighbourhood of the point $P_k$. 
Recall that from  the local forms \refE{locex} and \refE{Aset} 
of our basis elements we have 
in the neighbourhood of points $P_l$ with $l\ne k$ 
\begin{equation}\label{E:others}
H_{n,k}=O(z_l^{n+1}),\quad  
A_{n,k}=O(z_l^{n+1}), \quad
H_{(n,k)}=O(z_l^{n+1}).
\end{equation}
In  the
following, let $n\ne 0$. We have
\begin{equation}
\nw H_{(n,k)}=\nw(H_{0,k}\cdot A_{n,k})= \nw(H_{0,k})\cdot
A_{n,k}+H_{0,k}\;dA_{n,k}.
\end{equation}
The expression $\nw H_{0,k}$ is of nonnegative order, $A_{n,k}$ is of
order $n$, $H_{0,k}$ of order 0 and $dA_{n,k}$ of order $n-1$ at the point
$P_k$. Hence
\begin{equation}
\nw H_{(n,k)}=H_{0,k}\;dA_{n,k}+O(z_k^{n})dz_k.
\end{equation}
Now we compute
\begin{equation}
\ga(H_{(-1,k)},H_{(1,k)})= \sum_{i=1}^N\alpha_i\res_{P_i}
\tr(H_{(-1,k)}\cdot \nw
H_{(1,k)})
= \alpha_k\res_{P_k}
\tr(H_{(-1,k)}\cdot \nw
H_{(1,k)}).
\end{equation}
The last equality follows from the fact that by 
\refE{others} we do not have any poles at the points $P_l$ for $l\ne
k$. From the above it follows
\begin{equation}
(\al_k)^{-1}\ga(H_{(-1,k)},H_{(1,k)})
= \res_{P_k}\tr(H_{0,k} A_{-1,k} H_{0,k} dA_{1,k})=
\res_{P_k}\tr(H_{0,k}^2 \frac{dz_k}{z_k}).
\end{equation}
As $H^2_{0,k}=H^2+O(z_k)$ we obtain
\begin{equation}\label{E:con1}
(\al_k)^{-1}\ga(H_{(-1,k)},H_{(1,k)})= \res_{P_k}(\tr (H^2)\frac{dz_k}{z_k}) =\tr
(H^2)=\beta\cdot\kappa(H,H)\ne 0,
\end{equation}
with a non-vanishing constant $\beta$ relating the trace form
with the Cartan-Killing form. But
\begin{equation}\label{E:con2}
[H_{(-1,k)},H_{(1,k)}]= [H_{0,k}A_{-1,k}, H_{0,k}A_{1,k}]= 
[H_{0,k},H_{0,k}]A_{-1,k}A_{1,k}=0.
\end{equation}
The relations \refE{con1} and \refE{con2} are in contradiction to
\refE{contra}. 
\end{proof}

Now we are able to formulate the basic theorem.

\begin{theorem}\label{T:bounded}
$ $

(a) If $\g$ is simple (i.e. $\gb=\sln(n), \so(n), \spn(2n)$) then
the space of  bounded cohomology  classes is $N$-dimensional. If we
fix any connection form $\w$ then this space has as basis 
the classes of $\ga_{1,\w,C_i}$, $i=1,\ldots, N$. 
Every $\L$-invariant (with respect
to the connection $\w$) bounded  cocycle is a linear combination  of
the $\gamma_{1,\w,C_i}$.

(b) For $\laxg=\bgl$ the space of local cohomology classes which
are $\L$-invariant having been restricted to the scalar subalgebra
is 2N-dimensional. If we fix any connection form $\w$ then the
space has as basis the classes of the cocycles
$\gamma_{1,\w,C_i}$ and $\gamma_{2,C_i}$, $i=1,\ldots, N$. 
Every $\L$-invariant local
cocycle is a linear combination of the $\gamma_{1,\w,C_i}$ and
$\gamma_{2,C_i}$.
\end{theorem}
\begin{proof}[Proof of the theorem]
Here we only outline the proof. The technicalities are postponed until
Sections \ref{S:induction} and \ref{S:direct}.

By Propositions \ref{P:lsimp} and \ref{P:lcom} it follows that
$\L$-invariant and bounded cocycles are necessarily linear
combinations of the claimed form. This proves the theorem for the
cohomology space $\H_{b,\L}(\gb,\C)$. For the scalar subalgebra
we are done since we included the $\L$-invariance into the
conditions of the theorem. For semi-simple algebras we have to
show that there is an $\L$-invariant representative in each local
cohomology class. But by \refT{bounduni} the space
$\H_{b}(\gb,\C)$ is at most N-dimensional. As by
\refP{nonbound} no non-trivial   
linear combination of the  cocycles $\ga_{1,\w,C_i}$  
is a coboundary, this space is exactly N-dimensional and
$[\ga_{1,\w,C_i}]$ for $i=1,\ldots ,N$ constitute a basis.
\end{proof}
We conclude the following.
\begin{corollary}\label{C:uni}
Let $\g$ be a simple classical Lie algebra and $\gb$ the
associated Lax operator  algebra. Let $\w$ be a fixed connection
form. Then in each $[\ga]\in \H_{b}(\gb,\C)$ there exists a
unique representative $\ga'$ which is bounded and $\L$-invariant
(with respect to $\w$). Moreover, $\ga'=\sum_{i=1}^Na_i\ga_{1,\omega,C_i}$, with
$a_i\in\C$.
\end{corollary}

\medskip
\begin{proposition}\label{P:boundlinv}
$ $

\noindent (a) Let $\ga$ be a bounded  and $\L$-invariant cocycle
which is a coboundary, then $\ga= 0$.

\noindent (b) Let $\g$ be simple, then the cocycle $\ga_{1,\w',C_i}$
is $\L$-invariant with respect to $\w$, if and only if $\w=\w'$.
\end{proposition}
\begin{proof}

(a) By \refT{bounded}  we get $\ga=\sum_{i=1}^N
(\a_i\ga_{1,\w,C_i}+\b_i\ga_{2,C_i})$, with all $\b_i=0$
for the case $\g$ is simple. 
The summands constitute a basis of the cohomology. 
Hence, $\ga$ can only be a coboundary if all coefficients
vanish.

\noindent (b) As $\ga_{1,\w,C_i}$ and  $\ga_{1,\w',C_i}$ are local and
$\L$-invariant with respect to $\w$ their difference
$\ga_{1,\w,C_i}-\ga_{1,\w',C_i}$ 
is also local and $\L$-invariant. By
\refP{cind} it is  a coboundary. Hence by part (a)
$\ga_{1,\w,C_i}-\ga_{1,\w',C_i}= 0$. The relation \refE{cobw} gives the
explicit expression for the left hand side. Assume $\w\ne\w'$. Let
$m$ be the order of the element
\begin{equation}
\theta=\w-\w'=(\theta_mz_i^m+O(z_i^m))dz_i
\end{equation}
at the point $P_i$. As  $\g$ is simple the trace form $\tr(A\cdot
B)$ is  nondegenerate and  we find
\begin{equation}
\hat\theta=\hat\theta_{-m-1}z_i^{-m-1}+O(z_i^{-m}),
\end{equation}
such that $\b=\tr(\theta_m\cdot\hat\theta_{-m-1})\ne 0$. By
\refL{weak} we get $\hat\theta=[L,L']+L''$  with $\ord_{P_i}(L'')\ge
-m$. Hence,
\begin{multline}
0\ne \b=\tr(\theta_m\cdot\hat\theta_{-m-1})=
\frac{1}{2\pi\i}\int_{C_i}\tr\left((\w-\w')\cdot(
[L,L']+L'')\right)
\\
=\frac{1}{2\pi\i}\int_{C_i}\tr\left((\w-\w')\cdot [L,L']\right)
=\ga_{1,\w,C_i}(L,L')-\ga_{1,\w',C_i}(L,L')=0
\end{multline}
which is a contradiction.
\end{proof}

\medskip 
After these results which are valid for bounded cocycles we will
deduce
the corresponding classification theorem for local cocycles. In some
sense this is the main theorem of this article.
It will show for example that for Lax operator algebras associated to
simple Lie algebras there is up to rescaling
and equivalence only one non-trivial almost-graded central 
extension.
 
Recall the relation for the separating cycle
\begin{equation}\label{E:sepdecomp}
[C_S]=\sum_{i=1}^N[C_i]=-\sum_{j=1}^M
[C^*_j],
\end{equation}
and the corresponding relation for the cocycle obtained by integration.

\begin{theorem}\label{T:main}
$ $

(a) If $\g$ is simple (i.e. $\gb=\sln(n), \so(n), \spn(2n)$) then
the space of local cohomology  classes is one-dimensional. If we
fix any connection form $\w$ then this space will be generated by
the class of $\ga_{1,\w,C_S}$. Every $\L$-invariant (with respect
to the connection $\w$) local cocycle is a scalar multiple of
$\gamma_{1,\w,C_S}$.

(b) For $\laxg=\bgl$ the space of local cohomology classes which
are $\L$-invariant having been restricted to the scalar subalgebra
is two-dimensional. If we fix any connection form $\w$ then the
space will  be generated by the classes of the cocycles
$\gamma_{1,\w,C_S}$ and $\gamma_{2,C_S}$. Every $\L$-invariant local
cocycle is a linear combination of $\gamma_{1,\w,C_S}$ and
$\gamma_{2,C_S}$.
\end{theorem}
\begin{proof}
Let $\ga$ be a local cocycle. This says it is bounded from above and
from below.
For simplicity we abbreviate in this proof
\begin{equation}
\ga_{1,i}:=\ga_{1,\w,C_i}\quad
\ga_{2,i}:=\ga_{2,C_i},\quad
\ga_{1,j}^*:=\ga_{1,\w,C^*_j},\quad
\ga_{2,j}^*:=\ga_{2,C^*_j}.
\end{equation}
If we switch the role of  $I$ and $O$ 
we get an inverted almost-grading.
Every bounded from
below cocycle of the original grading, will get bounded from above
with respect to the inverted grading. Hence we can employ
\refT{bounded} in both directions and obtain
for the same cocycle two representations
\begin{equation}\label{E:l1}
\ga=\sum_{i=1}^N a_i\ga_{1,i}
+\sum_{i=1}^N b_i\ga_{2,i}
=-\sum_{j=1}^M a^*_j\ga^*_{1,j}
-\sum_{j=1}^M b^*_j\ga^*_{2,j}
,\quad \text{with}\quad 
 a_i,  a^*_j, b_i, b_j^*\in\C.
\end{equation}
If either $N=1$ or $M=1$ then via \refE{sepdecomp} the 
cocycle is obtained via integration over a separating cycle.
Hence the statement.

Otherwise both $N,M>1$.
By \refP{linind} the type (1) and type (2) cocycles 
are linearly independent, hence can be treated 
independently also in this context.
First consider type (1).
 Note that from \refE{sepdecomp} we get the
relation that 
\begin{equation}\label{E:bal}
\sum_{i=1}^N\ga_{1,i}=-
\sum_{j=1}^M\ga_{1,j}^*.
\end{equation}
Hence, 
\begin{equation}\label{E:l2}
\ga_{1,1}^*=-\sum_{i=1}^N\ga_{1,i}
- \sum_{j=2}^M\ga_{1,j}^*.
\end{equation}
{}From \refE{l1} we get 
\begin{equation}
0=
\sum_{i=1}^N a_i\ga_{1,i}+
\sum_{j=1}^M a^*_j\ga^*_{1,j}.
\end{equation}
If we plug \refE{l2} into this relation
we obtain
\begin{equation}\label{E:l3}
0=
( a_1- a_1^*)\sum_{i=1}^N\ga_{1,i}
+\sum_{i=2}^N( a_i- a_1)\ga_{1,i}
+\sum_{j=2}^M( a_k^*- a_1^*)\ga^*_{1,j}.
\end{equation}
Fix a $k$.
We take $X,Y\in\gb$ such that $\tr(XY)\ne 0$. By Riemann-Roch 
there exists $L', L\in\gb$ such that around the point
$P_k$ we have
\begin{equation}
L(z_k)= Xz_k+O(z^2_k),
\qquad
L'(z_k)= Yz_k^{-1}+O(z^0_k),
\end{equation}
both holomorphic at the points in $O$ and at $P_l$, $l\ne 1,k$.
The elements might have pole orders of sufficiently high degree at
$P_1$ to guarantee existence. The weak singularities will not
disturb.
By construction
\begin{equation}
\begin{gathered}
\ga_{1,k}(L,L')\ne 0,\qquad
\ga_{1,l}(L,L')=0,\quad l=2,\ldots, N, l\ne k,
\\
\ga_{1,j}^*(L,L')=0,\quad j=1,\ldots, M.
\end{gathered}
\end{equation}
Hence
\begin{equation}
\sum_{i=1}^N\ga_{1,i}(L,L')=
-\sum_{j=1}^M\ga_{1,j}^*(L,L')=0.
\end{equation}
If we plug $(L,L')$ into \refE{l3},
all terms in \refE{l3} will vanish, with the only
exception
\begin{equation}
0=( a_k- a_1)\ga_{1,k}(L,L').
\end{equation}
This shows $ a_k- a_1$ for all $k$. (In a similar way we get
 $ a_j^*- a_1^*$ for all $j$.) 
In particular, our cocycle we started with 
(resp. the $\ga_1$ part of it) 
is a multiple of the
cocycle obtained by integration over the separating cycle.
This was the claim.
The proof for the $\ga_2$ part works completely the same if
we take $X=Y$ a nonzero scalar matrix.
\end{proof}

\medskip

As in the bounded case we obtain also for the local case the
following corollary.
\begin{corollary}\label{C:unil}
Let $\g$ be a simple classical Lie algebra and $\gb$ the
associated Lax operator  algebra. Let $\w$ be a fixed connection
form. Then in each $[\ga]\in \H_{loc}(\gb,\C)$ there exists a
unique representative $\ga'$ which is local and $\L$-invariant
(with respect to $\w$). Moreover, $\ga'=a\ga_{1,\omega}$, with
$a\in\C$.
\end{corollary}
\section{Uniqueness of $\L$-invariant cocycles}
\label{S:induction}

\subsection{The induction step}
$ $

Recall from \refS{alm} the almost-graded structure 
of the Lax operator algebra $\gb$ and in
particular  the decomposition
$\laxg=\oplus_{n\in\Z}\laxg_n$ into subspaces of homogeneous
elements of  degree $n$. Also there the basis 
$\{L_{n,p}^u\mid u=1,\ldots,\dim\g, p=1,\ldots, N\}$
of the 
subspace $\laxg_n$ 
was introduced (see \refE{set}).

Let $\ga$ be an $\L$-invariant cocycle for the
algebra $\laxg$ which is bounded from above,
i.e. there exists an $R$ (independent of $n$ and $m$) such that
$\ga(\gb_n,\gb_m)\ne 0$ implies $n+m\le R$. Furthermore, we recall
that our connection $\omega$ needed to define the action of $\L$
on $\laxg$ is chosen to be holomorphic at the points in $I$.

For a pair $(L^u_{n,k},L^v_{m,t})$ of homogeneous elements  we call $n+m$
the {\it level} of the pair. We apply the technique  developed in
\cite{Scocyc}. We  will consider  cocycle values
$\ga(L^u_{n,k},L^v_{m,t})$ on pairs of level $l=n+m$ and will make
induction over the level. By the boundedness from above, the
cocycle values will vanish at all pairs of sufficiently high
level. It will turn out that everything will be fixed by the
values of the cocycle at level zero. Finally, we will show
that the cocycle  is a linear combination 
of the $N$ (resp. $2N$) basic cocycles as claimed in
\refT{bounded}.

For a cocycle $\ga$ evaluated for pairs of elements of level $l$
we will use the symbol $\equiv$ to denote that the expressions are
the same on both sides of an equation involving cocycle values
up to values of $\ga$  at
higher level. This has to be understood in the following strong
sense:
\begin{equation}
\sum \alpha^{(n,p,t)}_{(u,v)}\ga(L_{n,p}^u,L_{l-n,t}^v)\equiv 0,\qquad
 \alpha^{(n,p,t)}_{(u,v)}\in\C
\end{equation}
means a congruence modulo a linear combination of values of $\ga$
at pairs of basis elements of level $l'>l$. The coefficients of
that linear combination, as well as the  $\alpha^{(n,p,t)}_{(u,v)}$,
depend only on the structure of the  Lie algebra $\gb$ and do not
depend on $\ga$.

We will also use the same symbol $\equiv$ for equalities in  $\gb$
which are true modulo terms of higher degree compared to the terms
under consideration.

\medskip

By the $\L$-invariance we have
\begin{equation}
\ga(\nabla_{e_{k,r}}^{(\w)}L_{m,p}^u,L_{n,s}^v)
+ \ga(L_{m,p}^u,\nabla_{e_{k,r}}^{(\w)}L_{n,s}^v)=0.
\end{equation}
Using the almost-graded structure \refE{allstr} we obtain
(up to order $> (k+m+n)$) 
\begin{equation}\label{E:recform}
m\ga(L_{k+m,p}^u,L_{n,s}^{v})\,\delta_r^p+
n\ga(L_{m,p}^v,L_{n+k,s}^{v})\,\delta_r^s
\equiv 0,
\end{equation}
valid for all $n,m,k\in\Z$.

If in \refE{recform} all three indices $r,p$ and $s$ are different
then
the term on the left  hand side vanishes. 
If $r=p\ne s$ then we obtain
\begin{equation}
m\ga(L_{k+m,p}^u,L_{n,s}^{v})
\equiv 0.
\end{equation}
which is true for every $m$. Hence 
\begin{equation}\label{E:pds}
\ga(L_{l,p}^u,L_{n,s}^{v})
\equiv 0, \quad \text{for}\quad p\ne s\;.
\end{equation}
It remains $r=p=s$ and this yields
\begin{equation}\label{E:recforms}
m\ga(L_{k+m,s}^u,L_{n,s}^{v})+
n\ga(L_{m,s}^v,L_{n+k,s}^{v})
\equiv 0.
\end{equation}
\begin{proposition}\label{P:ln0}
Let $m+n\ne 0$ then at  level $m+n$ we have
\begin{equation}\label{E:ln0}
\ga(\laxg_m,\laxg_n)\equiv 0.
\end{equation}
\end{proposition}
\begin{proof}
From \refE{pds} we conclude that only elements with
the same second index could contribute in level $m+n$.
We put $k=0$ in \refE{recforms} and obtain
\begin{equation}
(m+n)\ga(L_{m,s}^u,L_{n,s}^{v})
\equiv 0,\quad \forall u,v.
\end{equation}
Hence if $(m+n)\ne 0$ the claim follows.
\end{proof}
\begin{proposition}\label{P:zerodeg}
\begin{equation}\label{E:zerodeg}
\ga(\gb_m,\gb_0)\equiv 0,\qquad \forall m\in\Z.
\end{equation}
\end{proposition}
\begin{proof}
We  evaluate \refE{recforms}
 for the values $m=1$ and $n=0$ and obtain the result.
\end{proof}
\begin{proposition}\label{P:zerobound}

\noindent (a) We have $\ga(\laxg_n,\laxg_m)=0$ if $n+m>0$, i.e.
the cocycle is bounded from above by zero.

\noindent (b) If  $\ga(\laxg_n,\laxg_{-n})=0$ then the cocycle
$\ga$ vanishes identically.
\end{proposition}
\begin{proof}
The proof stays word by word the same as in \cite{SSlax}. But 
as it is one of the central arguments and 
for
the convenience of the reader we repeat the arguments.
If $\ga=0$ there is nothing to prove. Assume $\ga\ne 0$. As $\ga$
is bounded from above, there will be a minimal upper bound $l$,
such that above $l$ all cocycle values will vanish. Assume that
$l>0$, then by \refP{ln0} the values at level $l$ are
expressions of levels bigger than $l$. But the cocycle vanishes
there. Hence it vanishes at level $l$ too. This is a
contradiction which proves (a).

By induction, using again  \refP{ln0} we obtain that if the
cocycle vanishes at  level 0, it vanishes everywhere. This proves
(b).
\end{proof}
Combining Propositions \ref{P:zerodeg} and \ref{P:zerobound} we
obtain
\begin{corollary}\label{C:zerodeg}
\begin{equation}\label{E:zerodg}
\ga(\gb_m,\gb_0)= 0,\qquad \forall  m\ge 0.
\end{equation}
\end{corollary}
\begin{proposition}
\begin{equation}\label{E:nm1}
\ga(L_{n,r}^u,L_{-n,s}^v)=n\cdot \ga(L_{1,r}^u,L_{-1,r}^v)\delta_r^s,
\end{equation}
\begin{equation}\label{E:p1m1}
\ga(L_{1,r}^u,L_{-1,s}^v)=\ga(L_{1,s}^v,L_{-1,r}^u)
\end{equation}
\end{proposition}
\begin{proof}
Assume $s\ne r$ then all expressions are of positive level and vanishes
by  \refP{zerobound}, hence the statement is true.
For $r=s$ we take
in \refE{recforms}  the values $n=-p$, $m=1$ and $k=p-1$.
This yields the expression \refE{nm1} up to higher level terms.
But as the level is zero, the higher level terms vanish. 
Setting
$n=-1$ in \refE{nm1} we obtain \refE{p1m1}.
\end{proof}

Independently of
the structure of the Lie algebra $\g$, we obtained the following
results for every $\L$-invariant and bounded cocycle $\ga$:
\begin{enumerate}
\item
The cocycle is bounded from above by zero.
\item
The cocycle is uniquely given by its values at level zero.
\item
At level zero the cocycle is uniquely fixed by its values
$\ga(L_{1,s}^u,L_{-1,s}^v)$, for $u,v=1,\ldots,\dim\g$ and
$s=1,\dots ,N$.
\item
The other  cocycle values at level zero are given by
$\ga(L_{n,s}^u,L_{-n,r}^v)=0$ if $s\ne r$,
$\ga(L_{0,s}^u,L_{0,s}^v)=0$ and  $\ga(L_{n,s}^u,L_{-n,s}^v)$ given by
\refE{nm1}.
\end{enumerate}

Let $X\in\g$ then we denote as always by $X_{n,s}$, $s=1,\ldots, N$ 
the element in $\laxg$ with
leading term $Xz_s^n$ at $P_s$
and higher orders at the other points in $I$. 
We define for $s=1,\ldots, N$
the maps
\begin{equation}\label{E:cocar}
\psi_{\ga,s}:\g\times\g\to\C\qquad \psi_{\ga,s}(X,Y):=
\ga(X_{1,s},Y_{-1,s}).
\end{equation}
Obviously,
$\psi_{\ga,s}$ is a bilinear form on $\g$.
\begin{proposition}
(a) $\psi_{\ga,s}$ is symmetric, i.e.
$\psi_{\ga,s}(X,Y)=\psi_{\ga,s}(Y,X)$.
\newline
(b)  $\psi_{\ga,s}$ is invariant, i.e.
\begin{equation}
\psi_{\ga,s}([X,Y],Z)=\psi_{\ga,s}(X,[Y,Z]).
\end{equation}
\end{proposition}
\begin{proof}
First we have by \refE{p1m1}
\begin{equation*}
\psi_{\ga,s}(X,Y)=\ga(X_{1,s},Y_{-1,s})= \ga(Y_{1,s},X_{-1,s})
=\psi_{\ga}(Y,X).
\end{equation*}
This is the symmetry. Furthermore, using
$[X_{1,s},Y_{0,s}]\equiv{[X,Y]}_{1,s}$, 
the fact that the
cocycle vanishes for positive level, and by the cocycle condition
we obtain
\begin{multline*}
\psi_{\ga,s}([X,Y],Z)=\ga({[X,Y]}_{1,s},Z_{-1,s})=
\ga([X_{1,s},Y_{0,s}],Z_{-1,s})=
\\
-\ga([Y_{0,s},Z_{-1,s}],X_{1,s}) -\ga([Z_{-1,s},X_{1,s}],
Y_{0,s}).
\end{multline*}
The last term vanishes due to \refC{zerodeg}. Hence
\begin{equation*}
\psi_{\ga,s}([X,Y],Z)=\ga(X_{1,s},[Y_{0,s},Z_{-1,s}])=
\ga(X_{1,s},{[Y,Z]}_{-1,s})= \psi_{\ga,s}(X,[Y,Z]).
\end{equation*}
\end{proof}

As the cocycle $\ga$ is fixed by the values
$\ga(L_{1,s}^u,L_{-1,s}^v)$, $s=1,\ldots,N$ 
 and they are fixed by the bilinear maps
$\psi_{\ga,s}$ we proved:
\begin{theorem}\label{T:fixing}
Let $\gamma$ be an $\L$-invariant cocycle for $\laxg$ which is
bounded from above.
Then $\ga$ it is bounded from above 
 by zero and  is completely fixed by the
associated symmetric and invariant bilinear forms $\psi_{\ga,s}$,
$s=1,\ldots, N$ on
$\g$ defined via \refE{cocar}.
\end{theorem}

\subsection{Simple Lie algebras $\g$}
$ $

By \refT{fixing} the $\L$-invariant 
cocycle $\ga$ is completely given by 
fixing the $N$-tuple $(\psi_{\ga,1},\psi_{\ga,2},\ldots,\psi_{\ga,N})$
of symmetric invariant bilinear forms $\psi_{\ga,s}$.
For a finite-dimensional simple Lie algebra every such form is a
multiple of the Cartan-Killing form $\kappa$. 
Hence the space of bounded cocycles is at most $N$-dimensional.
Our geometric cocycles
$\gamma_{1,\omega,C_i}$, see \refE{ga1}, for $i=1,\ldots,N$ 
are $\L$-invariant and
bounded  cocycles.
They are linearly independent, see \refP{linind}.
 Hence, we obtain that every bounded and
$\L$-invariant cocycle is a 
linear combination of the 
$\gamma_{1,\omega,C_i}$. 
Moreover,  they are a basis of the space of $\L$-invariant
and bounded cocycles.
By \refP{nonbound} they stay linearly independent 
after passing to cohomology 
and we obtain
\begin{proposition}\label{P:lsimp}
Let $\g$ be simple, then
\begin{equation}
\dim\H_{b,\L}(\gb,\C)=N,
\end{equation}
and this cohomology space is generated by the classes of
$\gamma_{1,\omega,C_i}$, $i=1,\ldots, N$. 
\end{proposition}

 \subsection{The case of $\laxgl(n)$}
$ $

We have the direct decomposition, as Lie algebras,
$\bgl=\laxs(n)\oplus\laxsl(n)$. Let $\ga$ be a cocycle of
$\laxgl(n)$ and denote by $\ga'$ and $\ga''$ its restriction to
$\laxs(n)$ and  $\laxsl(n)$ respectively.
As in \cite{SSlax} we obtain using   \refL{weak}
\begin{proposition}
\begin{equation}
\ga(x,y)=0,\quad \forall x\in \laxs(n),\ y\in \laxsl(n).
\end{equation}
\end{proposition}

Hence we can decompose the cocycle as 
$\ga=\ga'\oplus\ga''$. If $\ga$ is bounded/local and/or $\L$-invariant the
same is true for $\ga'$ and $\ga''$.

First we consider the algebra $\laxs(n)$. It is isomorphic to
$\A$, the isomorphism is given by
\begin{equation}
L\mapsto  \frac 1n\;\tr(L).
\end{equation}
In \cite[Thm. 5.7]{Scocyc} it was shown that 
the space of $\L$-invariant 
cocycles for $\A$ bounded from above is $N$-dimensional
and a basis is given by 
\begin{equation}
\gamma_i(f,g)=\cinc{C_i}fdg=\res_{P_i}(fdg),\quad \i=1,\ldots,N.
\end{equation}
Note that as $\A$ is abelian there do not exist non-trivial
coboundaries.
We obtain
\begin{equation}
\ga'(L,M)=\sum_{i=1}^N \alpha_i
\res_{P_i}(\tr(L)\cdot \tr(dM)) =
\sum_{i=1}^N\a_i\ga_{2,C_i}(L,M),
\end{equation}
by Definition \refE{g2}.

For the cocycle $\ga''$ of $\laxsl(n)$ we use \refP{lsimp} and
obtain $\ga''=
\sum_{i=1}^N\beta_i\gamma_{1,\omega,C_i}$. Altogether we showed
\begin{proposition}\label{P:lcom}
\begin{equation}
\dim\H_{b,\L}(\laxgl(n),\C)=2N.
\end{equation}
A basis is given by the classes of $\gamma_{1,\omega,C_i}$ and
$\gamma_{2,C_i}$, $i=1,\ldots, N$. 
\end{proposition}

\medskip
In this section we showed those parts of  \refT{bounded}
which deal with $\L$-invariant cocycles. 
In fact we showed the complete theorem  under the additional
assumption that our cohomology classes are  $\L$-invariant.
For the scalar part this is the best what could be expected.
Without $\L$-invariance there will be much more non-trivial
cohomology classes for the scalar algebra, see \cite{Scocyc} for
more information.
In the next section we will present a way how to get rid of this
condition for the simple Lie algebras.
\section{The simple case in general}
\label{S:direct}
In this section the 
Lax operator algebra $\gb$ is always based on a finite
simple classical Lie algebra.
As explained in the previous section if we put 
$\L$-invariance in the assumption then \refT{bounded} would 
have been 
proved. One way to complete the general proof is to 
 to show  that after
cohomological changes every bounded cocycle has also a bounded
$\L$-invariant representing  it.
In fact, we will do this. But unfortunately, we do not have 
a direct proof. Instead, 
by a quite different approach we will show  that
for the  simple Lie algebra case 
the space of bounded  cohomology classes (of the Lax operator
algebras)
is at most $N$-dimensional without assuming  $\L$-invariance a
priori. Combining this result with the result of the last section
that the space of
$\L$-invariant bounded cohomology classes is $N$-dimensional we see
that in the simple case each bounded cohomology class is
automatically $\L$-invariant. Moreover, we showed there that it
has a unique $\L$-invariant representing cocycle which is given as
linear combination  of $\ga_{1,\w,C_i}$, $i=1,\ldots,N$.

The theorem we are heading for is
\begin{theorem}\label{T:bounduni}
Let $\g$ be a simple classical Lie algebra over $\C$ and $\gb$ the
associated Lax operator algebra with its almost-grading. Every
bounded cocycle on  $\gb$ is cohomologous 
to a  distinguished  cocycle which is bounded from above by zero. 
The space of distinguished cocycles is at most $N$-dimensional.
\end{theorem}
\begin{remark}
What we will show is  the following. 
Every cocycle  bounded from above is
cohomologous to a cocycle which is fixed by its value at  $N$
special pair of elements in $\gb$ (namely by
$\ga(H^\a_{1,s},H^{\a}_{-1,s})$ for one fixed simple root $\a$, see
below for the notation). Besides the structure of
$\g$ we only use the almost-gradedness of $\gb$ with leading terms
given in \refE{chevlax}.
\end{remark}

The presentation is quite similar to \cite{SSlax}. Those proofs
which are completely of the same structure will not be repeated
here.

First we need  to recall some facts about the Chevalley
generators of $\g$. Choose a root space decomposition
$\g=\fh\oplus_{\alpha\in\Delta}\g^\alpha$. As usual $\Delta$
denotes the set of all roots $\alpha\in\fh^*$. Furthermore, let
$\{\alpha_1,\alpha_2,\ldots,\alpha_p\}$ be a set of simple roots
($p=\dim\fh$). With respect to this basis, the root system splits
into positive and negative roots, $\Delta_+$ and $\Delta_-$
respectively. With $\a$  a positive root, $-\a$ is a negative root
and vice versa. For $\a\in\Delta$ we have $\dim\g^\alpha=1$.
Certain elements $E^\a\in\g^\a$ and $H^\a\in \fh$, $\a\in\Delta$
can be fixed so that for every positive root $\a$
\begin{equation}\label{E:chev}
[E^\a,E^{-\a}]=H^\a,\qquad [H^\a,E^{\a}]=2E^{\a},\qquad
[H^\a,E^{-\a}]=-2E^{-\a}.
\end{equation}
We use also $H^i:=H^{\a_i}$, $i=1,\ldots,p$ for the elements
assigned to the simple roots. A vector space basis, the Chevalley
basis, of $\g$ is given by $\{E^\a, \alpha\in\Delta;\  H^i, 1\le
i\le p\}$.

We denote by $(\,\, ,\,\,)$ the inner product on $\fh^*$ induced by
the Cartan-Killing form of $\g$. The following relations hold
\begin{equation}\label{E:chevfull}
\begin{aligned}\
[H^\a,H^\b]&=0,\quad
\\
[H^\a,E^{\pm\b}]&=\pm2\frac{(\b,\a)}{(\b,\b)}E^{\pm \a},
\\
[H,E^{\a}]&=\a(H)E^\a,\quad H\in\fh,
\\
[E^\a,E^{\b}]&=
\begin{cases}
H^\a, &\a\in\Delta_+,\ \b=-\a,
\\
-H^\a, &\a\in\Delta_-,\ \b=-\a,
\\
\pm(r+1)E^{\a+\b},&\a,\ \b,\ \a+\b\in\Delta,
\\
0,&\text{otherwise.}
\end{cases}
\end{aligned}
\end{equation}
Here $r$ is the largest nonnegative integer such that $\a-r\b$
still is a root.

\medskip
As in the other parts of this article, we denote by $E_{n,s}^\a$,
$H_{n,s}^\a$ the unique elements in $\laxg_n$ (i.e. of degree $n$) 
for which the
expansions at $P_s$ start with $E^\a z_s^n$ and $H^\a z_s^n$ respectively,
and at the Points $P_i\in I$, $i\ne s$ it is of higher order.

The following elements form a basis of $\laxg$:
\begin{equation}\label{E:laxb}
\{\;E_{n,s}^\a, \alpha\in\Delta;\  H_{n,s}^i, 
1\le i\le p\mid\ n\in\Z, s=1,\ldots, N\;\}.
\end{equation}
The structure equations, up to higher degree terms, are
\begin{equation}\label{E:chevlax}
\begin{aligned}\
[H_{n,s}^\a,H_{m,r}^\b]&\equiv0,\quad
\\
[H_{n,s}^\a,E_{m,r}^{\pm\b}]
&\equiv\pm2\frac{(\b,\a)}{(\b,\b)}E_{n+m,r}^{\pm\b}
\,\delta_r^s,
\\
[H_{n,s},E_{m,r}^{\a}]&\equiv\a(H)E_{n+m,r}^\a
\,\delta_r^s
,\quad H\in\fh,
\\
[E_{n,s}^\a,E_{m,r}^{\b}]&\equiv
\begin{cases}
H_{n+m,s}^\a
\,\delta_r^s
, &\a\in\Delta_+,\ \b=-\a,
\\
-H_{n+m,s}^\a
\,\delta_r^s,
 &\a\in\Delta_-,\ \b=-\a,
\\
\pm(r+1)E_{n+m,s}^{\a+\b}
\,\delta_r^s,
,&\a,\ \b,\ \a+\b\in\Delta,
\\
0,&\text{otherwise.}
\end{cases}
\end{aligned}
\end{equation}
Recall that the symbol $\equiv$ denotes equality up to elements of
degree higher than the sum of the degrees of the elements under
consideration. Here,  the elements not written down  are elements
of degree $>n+m$. Also recall that by the almost-gradedness there
exists a $S$, independent of $n$ and $m$, such that only elements
of degree $\le n+m+S$ appear.

\medskip

Let $\ga'$ be a  cocycle for $\laxg$ which is bounded from above.
For the elements in $\g$ we get
\begin{equation}
E^{\pm\a}=\pm 1/2[H^\a,E^{\pm\a}], \quad H^i=[E^{\alpha_i},
E^{-\alpha_i}], \ i=1,\ldots, p.
\end{equation}
Consequently, for $\laxg$ we obtain
\begin{equation}\label{E:ehn}
\begin{aligned}
E_{n,s}^{\pm\a}&=\pm 1/2[H_{0,s}^\a,E_{n,s}^{\pm\a}]+Y(n,s,\a),
\\
H_{n,s}^i&=[E_{0,s}^{\alpha_i}, E_{n,s}^{-\alpha_i}]+Z(n,s,i), \ i=1,\ldots, p.
\end{aligned}
\end{equation}
where $Y(n,s,\a)$ and $Z(n,s,i)$ are sums of elements of degree
between $n+1$ and $n+S$. Fix a number  $M\in Z$ such that the
cocycle $\ga'$ vanishes for all levels $\ge M$. We define a linear
map $\Phi:\laxg\to\C$ by (descending) induction on the degree of
the basis elements \refE{laxb}. First
\begin{equation}\label{E:ehni}
\Phi(E_{n,s}^{\a}):=\Phi(H_{n,s}^i):=0,\qquad \a\in\Delta,\
i=1,\ldots,p,
\ s=1,\ldots, N
\quad n\ge M.
\end{equation}
Next we define inductively ($\a\in\Delta_+$, $s=1,\ldots, N$)
\begin{equation}
\begin{aligned}\label{E:bel}
\Phi(E_{n,s}^{\pm \a})&:=\pm 1/2\ga'(H_{0,s}^\a,E_{n,s}^{\pm
a})+\Phi(Y(n,s,\pm\a)),
\\
\Phi(H_{n,s}^i)&:=\ga'(E_{0,s}^{\a_i},E_{n,s}^{-\a_i})+\Phi(Z(n,s,i)).
\end{aligned}
\end{equation}
The cocycle  $\ga=\ga'-\delta\Phi$ is cohomologous to the original
cocycle $\ga'$. As $\ga'$ is bounded from above, and,  by
definition, $\Phi$ is also bounded from above, the cocycle $\ga$
is  bounded from above too.

By the construction of $\Phi$ we have
$\Phi([H_{0,s}^\a,E_{n,s}^{\pm\a}]=\ga'(H_{0,s}^\a,E_{n,s}^{\pm\a})$ and
$\Phi([E_{0,s}^{\a_i},E_{n,s}^{-\a_i}])=\ga'(E_{0,s}^{\a_i},E_{n,s}^{-\a_i})$.
Hence
\begin{proposition}\label{P:p1}
\begin{equation}\label{E:p1}
\begin{gathered}
\ga(H_{0,s}^\a,E_{n,s}^{\pm\a})=0,\quad
\ga(E_{0,s}^{\a_i},E_{n,s}^{-\a_i})=0,
\\
\a\in\Delta_+,\
i=1,\ldots,p,\ s=1,\ldots, N,\quad n\in\Z.
\end{gathered}
\end{equation}
\end{proposition}
\begin{definition}\label{D:norc}
A cocycle $\ga$ is called \emph{normalized} if it fulfills
\refE{p1}. 
\end{definition}

By the above construction we
showed that every cocycle bounded from above is
cohomologous to a normalized one, which is also bounded from
above. In the following we assume that our cocycle is already
normalized.

\begin{proposition}\label{P:summary}
Let $\a_1$ be a fixed simple root, $\alpha$ and $\beta$ arbitrary
roots
 and $\ga$ a normalized cocycle, then for all $s,r=1,\ldots, N$,  
 $n,m\in\Z$
we have
\noindent
\begin{equation}
\begin{aligned}
\ga(E_{m,s}^\a,H_{n,r})&\equiv 0, \qquad\qquad  H\in\fh,\  \a\in\Delta
\\
\ga(E_{m,s}^\a,E_{n,r}^\b)&\equiv 0, \qquad \qquad \a,\b\in\Delta, \
\b\ne-\a,
\\
\ga(E_{m,r}^\a,E_{n,s}^{-\a})&\equiv
u\ga(H_{m,r}^{\a_1},H_{n,r}^{\a_1})\,
\delta_r^s, \qquad
\a\in\Delta,
\\
\ga(H_{m,r}^\a,H_{n,s}^{\b})&\equiv
t\ga(H_{m,r}^{\a_1},H_{n,r}^{\a_1})
\delta_r^s,, \qquad
\a,\b\in\Delta_+,
\end{aligned}
\end{equation}
with $u,t\in\C$.
\begin{equation}\label{E:hn0}
\ga(H_{n,r}^{\a_1},H_{0,s}^{\a_1})\equiv 0,
\end{equation}
\begin{equation}\label{E:hnrec}
\ga(H_{n+1,s}^{\a_1},H_{l-(n+1),r}^{\a_1})\equiv
\left(\ga(H_{n-1,s}^{\a_1},H_{l-(n-1),s}^{\a_1})+ 2\ga(H_{1,s}^{\a_1},
H_{l-1,s}^{\a_1)}\,
\right)\delta_r^s.
\end{equation}
For a simple root $\a_1$ and for a level $l\ne 0$ we have
\begin{equation}\label{E:hnz}
\ga(H_{n,r}^{\a_1},H_{l-n,s}^{\a_1})\equiv 0.
\end{equation}
\end{proposition}
\begin{proof}
In the two point case the statement of the proposition
consists of a chain of individual statement which were
proved in \cite{SSlax}. In fact, the proofs presented
there remain valid if one just adds in all relations there
for the Lie algebra elements $Y_n$ the second index to obtain 
$Y_{n,s}$.
By the almost-graded structure, resp. 
its fine structure \refE{alalg1} for the expressions $[Y_{n,s},Z_{m,r}]$
in the relations only terms involving $s=r$ will contribute on the
level under considerations. If $s\ne r$ they will contribute
only to higher level. 
Hence, all relations there can be read with respect to all the second
indices the same up to higher level.
Hence, the proof is completely  analogous.
\end{proof} 
\begin{proposition}
Let $\ga$ be a normalized cocycle. Then
\newline
(a) it vanishes for levels greater than zero, i.e. 
\begin{equation}
\ga(\gb_n,\gb_n)=0, \quad\text{for} \ n+m>0.
\end{equation}
(b) All levels $l<0$ are fixed by the level zero.
\end{proposition}
\begin{proof}
By the 
propositions above we showed that the expressions at level $l$ of
the cocycle can be reduced to expressions of  levels $> l$ and
values $\ga(H_{n,r}^\a,H_{l-n,r}^\a)$. As long as the level is $\ne 0$,  by
\refE{hnz} also these values can be expressed by higher level.
Hence by induction, starting with the upper bound of the cocycle,
we obtain that the upper bound for the level of the cocycle values
is equal to zero. Also it follows that the values at levels $l<0$
are fixed by induction by the values at level zero.
\end{proof}
 Hence it
remains to consider the level zero. 
\begin{proposition}\label{P:p51}
Let $\a$ be a simple root. At level $l=0$ the cocycle values 
for $s=1,\ldots, N$ are
given by the relations
\begin{equation}\label{E:cat0}
\ga(H_{n,s}^{\a},H_{-n,r}^\a)= n\cdot
  \ga(H_{1,s}^{\a},H_{-1,s}^\a)
\,\delta_s^r,\qquad
\ga(H_{0,r}^{\a},H_{0,s}^{\a})=0.
\end{equation}
\end{proposition}
\begin{proof}
If we set the value $l=0$ in \refE{hnrec}  we obtain the relation
\begin{equation}
\ga(H_{n+1,s}^\a,H_{-(n+1),r}^\a)\equiv 
\left(
\ga(H_{n-1,s}^\a,H_{-(n-1),s}^\a)+
2\ga(H_{1,s}^\a,H_{-1,s}^\a)
\right)\cdot\delta_r^s.
\end{equation}
As all cocycle values of level $l>0$ are vanishing we can
replace $\equiv$ by $=$.
Now the claimed expression follows.
\end{proof}

\begin{proof}[Proof of \refT{bounduni}]
After adding a suitable coboundary we might replace the given
$\ga$ by a normalized one. 
Using \refP{p1}, \ref{P:summary}, and \ref{P:p51}
everything depends
only on the values $\ga(H_{1,s}^{\a},H_{-1,s}^\alpha)$, $s=1,\ldots,
N$
 for one (fixed)
simple root. This proves that there are at most $N$ linearly
independent
normalized cocycles.
\end{proof}

\begin{proposition}
If a normalized cocycle $\ga$ is a coboundary then it vanishes
identically.
\end{proposition}
\begin{proof}
As explained above, a normalized cocycle is fixed by the values
$\ga(H_{1,s}^{\a},H_{-1,s}^\alpha)$. 
We set
\begin{equation}
H_{(1,s)}^\a:=H_{0,s}^\a A_{1,s}\equiv
H^\a_{1,s},\quad \text{ and }\quad 
H_{(-1,s)}^\a:=H_{0,s}^\a A_{-1,s}\equiv H^\a_{-1,s}. 
\end{equation}
Hence
\begin{equation}
[H_{(1,s)}^\a,H_{(-1,s)}^\a]=[H_{0,s}^\a,H_{0,s}^\a]A_{1,s}A_{-1,s}=0.
\end{equation}
As the cocycle vanishes for positive levels, and as
$\ga=\delta\phi$ is assumed to be a coboundary we get
\begin{equation}
\ga(H_{1,s}^{\a},H_{-1,s}^\alpha)= \ga(H_{(1,s)}^{\a},H_{(-1,s)}^\alpha)
=\phi([H_{(1,s)}^\a,H_{(-1,s)}^\a])=\phi(0)=0.
\end{equation}
Hence, all cocycle values are zero, as claimed.
\end{proof}

\appendix
\section{Example  $\gl(n)$}\label{S:example}
In this appendix 
we will reproduce  as an illustration for the reader the proof
that the product of two Lax operators for the algebra $\gl(n)$ 
is again a Lax operator. This means that the equations
\refE{gldef} are fulfilled for their product. 
Hence, $\glb(n)$ will be closed under
commutator too.
This result is due to Krichever and Sheinman  \cite{KSlax}.
In a similar manner the other cases are treated (but now only for the
commutators). Furthermore,  it is shown that the connection 
operators $\nabla_{e}^{(\omega)}$  act indeed on $\gb$.
The original proofs (involving 
partly tedious calculations) can be found in \cite{KSlax}, \cite{SSlax}, 
\cite{SheBo}.

The singularities at the points in $A$ are not  
bounded. Hence, they will not create problems and the proofs need  only to 
consider the weak singular points. Consequently, the statements are
also true in the multi-point case.

\bigskip

We start with two elements $L'$ and $L''$ with corresponding  expansions 
\refE{glexp} and examine their product $L=L'L''$.
 For this we have to consider each point  
$\gamma_s$  (with local coordinate $w_s$) of
 the weak singularities with $\a_s\ne 0$ 
separately. 
Taking into account only those parts which might contribute
we obtain for $L$
\begin{multline}
L=\frac {L'_{s,-1}L''_{s,-1}}{w_s^2}
+
\frac {L'_{s,-1}L''_{s,0}
+L'_{s,0}L''_{s,-1}}{w_s^1}
\\
+\left({L'_{s,-1}L''_{s,1}}+
{L'_{s,0}L''_{s,0}}+
{L'_{s,1}L''_{s,-1}}\right)+ O(w_s^1).
\end{multline}
By expanding  the first numerator
we get
\begin{equation}
L'_{s,-1}L''_{s,-1}=\al_s{\beta_s'}^t\al_s{\beta_s''}^t=0
\end{equation}
as ${\beta_s'}^t\al_s=0 $ by \refE{gldef}.
Hence, there is no pole of order two appearing.

Next we consider the expression which comes with pole order one.
\begin{equation}
L_{s,-1}=
L'_{s,-1}L''_{s,0}+L'_{s,0}L''_{s,-1}
= \alpha_s{\beta_s'}^tL''_{s,0}+
L_{s,0}'\alpha_s{\beta_s''}^t.
\end{equation}
As by the conditions $L'_{s,0}\alpha_s=\kappa_s'\alpha_s$
we can write
\begin{equation}
L_{s,-1}=\alpha_s\beta_s^t,\qquad
\text{with}\quad 
\beta_s^t={\beta_{s}'}^tL_{s,0}''+\kappa_s'{\beta_s''}^t.
\end{equation}
For the trace condition we obtain
\begin{equation}
\tr(L_{s,-1})=({\beta_{s}'}^tL_{s,0}''+\kappa_s'{\beta_s''}^t)
\alpha_s=
\kappa_s''{\beta_s'}^t\alpha_s+
\kappa_s'{\beta_s''}^t\alpha_s=0.
\end{equation}
Hence, we have the required form.

Finally we have to verify that $\alpha_s$ is an eigenvector of
$L_{s,0}$. First we note that $L_{s,-1}''\alpha_s=0$ and 
$L_{s,0}'L_{s,0}''\alpha_s=\kappa_s'\kappa_s''\alpha_s$. Also
\begin{equation}
L_{s,-1}'L_{s,1}''\alpha_s=\alpha_s({\beta_s'}^tL_{s,1}''\alpha_s)
=({\beta_s'}^tL_{s,1}''\alpha_s)\alpha_s.
\end{equation}
Hence, indeed $\alpha_s$ is an eigenvector with eigenvalue
${\beta_s'}^tL_{s,1}''\alpha_s+\kappa_s'\kappa_s''$. This shows the
claim
that $L\in\glb(n)$.


\end{document}